\documentclass{amsart}
\usepackage[utf8]{inputenc}
\usepackage{zeyu}
\usepackage{calc}

\makeatletter
\@for\@thmenv:={theorem,lemma,corollary,proposition,conjecture,problem,definition,example,remark}\do{%
  \expandafter\let\csname theH\@thmenv\endcsname\thetheorem}
\makeatother

\makeatletter
\def\paragraph{\@startsection{paragraph}{4}%
  \z@\z@{-\fontdimen2\font}%
  {\normalfont\bfseries}}
\makeatother

\title{On Counting Independent Sets in Regular Hypergraphs}

\author{Michail Sarantis}
\address{Skyserv Handling Services, Athens, Greece}
\email{msarantis@hotmail.com}

\author{Prasad Tetali}
\address{Department of Mathematical Sciences, Carnegie Mellon University, Pittsburgh, PA 15213, USA; research supported in part by the Alexander M. Knaster Professorship and the NSF grant DMS-2151283}
\email{ptetali@cmu.edu}

\author{Zeyu Zheng}
\address{Department of Mathematical Sciences, Carnegie Mellon University, Pittsburgh, PA 15213, USA}
\email{zeyuzhen@andrew.cmu.edu}

\begin{document}

\begin{abstract}
Balogh, Bollob\'as and Narayanan conjectured that among all finite simple \(r\)-uniform \(d\)-regular hypergraphs, the number of weak independent sets is maximized by a natural quasi-bipartite construction \(H_{r,d}\).  We give three types of evidence for this conjecture.  For every fixed \(r\), we prove the conjectured asymptotic exponential rate whenever the twin quotient has maximum pair codegree \(o(d)\).  The proof uses the hypergraph container method.  For hypergraphs with no cross-edges, the occupancy method gives the sharper error bound \(O_r(\log d/d)\). We show that a stronger version of the conjecture in terms of the so-called occupancy fraction is {\em not true}, by providing a counterexample for every $r\ge 3$.  We also prove exact cases of the conjecture when the hypergraph is $2$-regular. Using an entropy decomposition in the dual edge-cover problem, we settle every odd \(r\) and the case \((r,d)=(4,2)\).
\end{abstract}

\maketitle

\section{Introduction}
In a graph \(G=(V,E)\), an independent set \(I\) is a set of vertices such that no two vertices are connected. Let $\cI=\cI(G)$ denote the collection of all independent sets of $G$ and $i(G):=|\cI|$. The problem of maximizing the number of independent sets in a regular graph has a long history. Kahn \cite{kahn2001entropy} proved that disjoint copies of \(K_{d,d}\) maximize the number of independent sets over all \(d\)-regular bipartite graphs, while Zhao \cite{zhao2010number} extended this result to all \(d\)-regular graphs.  In particular, if \(G\) is a \(d\)-regular graph on \(n\) vertices, then
    \[
        \frac{\log i(G)}{n}\leq \frac{\log i(K_{d,d})}{2d}.
    \]
Sah, Sawhney, Stoner and Zhao \cite{sah2019number} proved a general upper bound without any assumption on the structure of \(G\), resolving the extremal problem in terms of the degrees of the vertices. Namely, they showed that if \(G\) is a graph with no isolated vertices, then
    \[
        i(G)\leq \prod_{uv\in E(G)} i(K_{d_u,d_v})^{1/(d_ud_v)}\,,
    \]
where \(d_v\) is the degree of the vertex \(v\) in \(G\).

Progress on the analogous hypergraph problem is much more limited.  Before we proceed to the details, we recall some basic definitions.  Throughout the paper all hypergraphs are finite and simple.  A hypergraph is called \(r\)-uniform if every hyperedge contains precisely \(r\) vertices, and \(d\)-regular if every vertex belongs to exactly \(d\) edges.  We also call a hypergraph linear if every two hyperedges intersect in at most one vertex.

In the setting of hypergraphs, there is more than one natural extension of the definition of independent sets.  If \(G=(V,E)\) is a hypergraph, one may call \(I\subseteq V\) independent if no two vertices of \(I\) belong to the same edge; another natural definition is to call \(I\) independent if it does not contain an edge in its entirety.  In the first case, we are talking about \emph{strong} independent sets, while in the second we are talking about \emph{weak} independent sets.  Notice that the former can be transformed into independent sets in graphs by replacing each hyperedge with a clique, while counting weak independent sets is a genuinely hypergraph problem.  For \(r\)-uniform hypergraphs, we use the notation \(i_2(G)\) for strong independent sets and \(i_r(G)\) for weak independent sets. One may similarly define $i_j(G)$ for sets that contain less than $j$ vertices from each hyperedge. When there is no danger of confusion, we simply write \(i(G)\).

For linear hypergraphs, the problem was studied for strong independent sets by Ordentlich and Roth \cite{ordentlich2004independent}, while 
Balobanov and Shabanov \cite{balobanov2018number} showed bounds for $i_j(G)$ for all $2\leq j \leq r$. In recent work, Cohen, Perkins, Sarantis and Tetali \cite{cohen2022number} obtained asymptotically sharp bounds for weak independent sets in $r$-uniform, $d$-regular linear hypergraphs with no cross-edges i.e. no linear triangles. More specifically, they proved that for any $n$-vertex hypergraph satisfying the above,
\begin{equation}\label{eqn:cpst}
    \frac{\log i_r(G)}{n}\leq \frac{r-1}{r}+O\left(d^{-1/(r-1)}\right).
\end{equation}
Their proof is via the occuppancy method of Davies, Jenssen, Perkins and Roberts \cite{davies2017independent}, which we will discuss in more detail (and use) in \Cref{sec:occupancy}.

Shortly afterwards, Balogh, Bollob\'as and Narayanan \cite{balogh2021counting} constructed the following hypergraph: Consider $rd$ vertices and mark $d$ of them. Partition the remaining \((r-1)d\) vertices into \(d\) parts of size \(r-1\), and take as edges all unions of one marked vertex with one part. Denote this hypergraph by $H_{r,d}$. Then,
\begin{conjecture}[Balogh--Bollob\'as--Narayanan]\label{conj:bbn}
If \(G\) is any \(r\)-uniform, \(d\)-regular hypergraph on \(n\) vertices, then
\[
    i_r(G)\leq i_r(H_{r,d})^{n/(rd)},
\]
or equivalently,
\[
    \frac{\log i_r(G)}{n}\leq \frac{\log i_r(H_{r,d})}{rd}.
\]
\end{conjecture}
They proved the conjecture for a certain family of graphs, which they called quasi-bipartite, containing the extremal construction.

\subsection*{The hard-core model and the occupancy fraction}
To state our main results towards \Cref{conj:bbn}, we need some extra definitions. For a hypergraph \(G\), let \(\mathcal I(G)\) be the family of weak independent sets, and define its independence polynomial as
$$Z_G(\lam)=\sum_{I\in\mathcal I(G)}\lam^{|I|}.$$
The \textit{hard-core model} with fugacity $\lam\geq0$ is the distribution on independent sets on $G$ for which
$$\PP(\mathbf{I}=I)=\frac{\lam^{|I|}}{Z_G(\lam)}.$$
The occupancy fraction $\alpha_G(\lam)$ is the expected fraction of vertices in a randomly chosen independent set according to the hard-core model, i.e.
$$\alpha_G(\lam)=\frac{\E|\mathbf{I}|}{n}.$$
A straightforward computation shows that $\E|\mathbf{I}|=\frac{\lam Z'_G(\lam)}{\lam}$, hence $\alpha_G(\lam)=\frac{\lam (\log Z_G(\lam))'}{n}$ and
\begin{equation}\label{eqn:occ_int}
    \frac{\log Z_G(\lam)}{n}=\int_0^\lam \frac{\alpha_G(t)}{t}dt.
\end{equation}
Since $i(G)=Z_G(1)$ and due to \eqref{eqn:occ_int}, there is a clear set of implications. An inequality between occupancy fractions uniformly for $t>0$ implies a corresponding inequality of the partition functions. In turn, this implies an inequality for the number of independent sets. 

For graphs, $K_{d,d}$ is indeed the maximizer of all these quantities, and the proofs progressed from the weakest to the strongest statement \cite{kahn2001entropy, galvin2004weighted, davies2017independent}. For matchings in graphs and the monomer-dimer model, the occupancy fraction statement was proved directly to imply the rest \cite{davies2017independent}. It is natural to state the corresponding generalizations of \Cref{conj:bbn}.
\begin{conjecture}\label{conj:bbn_general}
    If \(G\) is any \(r\)-uniform, \(d\)-regular hypergraph on \(n\) vertices, then
    \begin{enumerate}[(i)]
        \item $\alpha_G(\lam)\leq \alpha_{H_{r,d}}(\lam)$ for all $\lam\geq0$.
        \item $\frac{\log Z_G(\lam)}{n}\leq \frac{\log Z_{H_{r,d}}(\lam)}{rd}$ for all $\lam\geq0$.
        \item $\frac{\log i_r(G)}{n}\leq \frac{\log i_r(H_{r,d})}{rd}.$
    \end{enumerate}
\end{conjecture}
As we discussed, $(i)\Rightarrow (ii)\Rightarrow (iii)$, and in fact the first two imply the third as long as they hold for $\lam\in[0,1].$

\subsection*{Our results}
We begin with a rather unexpected counterexample. Let $\widetilde{H}_{r,d}$ be the hypergraph obtained by removing a perfect matching from \(H_{r,d}\). This is an $r$-uniform, $(d-1)$-regular hypergraph.
\begin{proposition}\label{prop:counterex}
    \Cref{conj:bbn_general} (i) is false for every $r\geq 3$. More specifically, for every $r\geq 3$ there exists a positive integer $d_0:=d_0(r)$ such that
    $$\alpha_{\widetilde{H}_{r,d}}(1)>\alpha_{H_{r,d-1}}(1)\,,$$
    for all $d\geq d_0$.
\end{proposition}
The proof is deferred to \Cref{sec:appendix-counterexample}. This doesn't refute the other conjectures, but shows they can't be derived by a uniform bound on the occupancy fraction. In other words, the application of the occupancy method as is can't yield \Cref{conj:bbn_general} (ii), even in our restricted family of hypergraphs. We still believe it to be true.

In light of \Cref{prop:counterex}, we will pursue asymptotic versions in the spirit of \eqref{eqn:cpst}. For \(H_{r,d}\), a straightforward computation gives
\[
    Z_{H_{r,d}}(\lam)
    =(1+\lam)^{(r-1)d} +\big((1+\lam)^d-1\big)
    \big((1+\lam)^{r-1}-\lam^{r-1}\big)^d.
\]
In particular, taking \(\lam=1\),
\begin{equation}\label{eqn:Hrd_asympt}
    \frac{\log i_r(H_{r,d})}{rd}
    =\frac{\log\Big(2^{(r-1)d}+(2^d-1)(2^{r-1}-1)^d\Big)}{rd}
    =1+\frac{\log\left(1-2^{-(r-1)}\right)}{r}
    +O_r\left(\frac{\rho_r^d}{d}\right),
\end{equation}
where \(\rho_r=\big(2(1-2^{-(r-1)})\big)^{-1}<1\) for \(r\geq3\).  Note that the constant \(1+\log\left(1-2^{-(r-1)}\right)/r\) is between \((r-1)/r\) and \(1\).  This should be compared with the sharper asymptotic estimates that are available under the linearity assumption of \eqref{eqn:cpst}.

For some of our results, we will retain the cross-edge-free assumption of \cite{cohen2022number}, modified accordingly for non-linear settings. We define a cross-edge as a triple of distinct hyperedges \(e,f,g\) for which there exist vertices \(u,v,w\), not necessarily distinct, such that \(u\in e\cap f\), \(v\in f\cap g\) and \(w\in g\cap e\), and at least one of \(u,v,w\) is not in \(e\cap f\cap g\). The condition excludes all nondegenerate Berge triangles and, since the witnesses need not be distinct, some additional overlapping triples of hyperedges.  In the linear case it is exactly Berge-triangle-freeness, a standard forbidden configuration in the extremal theory of Berge hypergraphs \cite{gerbner2017extremal, gyHori2006triangle}.  Thus the class contains the \(r\)-uniform \(d\)-regular hypergraphs of Berge girth at least four, and such regular high-girth hypergraphs form a nontrivial family studied through Moore-type bounds and cage problems \cite{ellis2013regular,erskine2022small}.  Unlike linearity or a high-girth assumption, the condition also allows many hyperedges to share a large common set of vertices, as happens in \(H_{r,d}\).

Our first main result is an asymptotic version of \Cref{conj:bbn_general} (ii) on bounded intervals. 

\begin{theorem}\label{thm:main_container}
    Fix \(r\geq3\) and \(L>0\).  Let \(G\) be an \(r\)-uniform, \(d\)-regular hypergraph on \(n\) vertices with no cross-edges, where \(d\geq2\). Then, uniformly for \(0\leq\lam\leq L\),
\[
    \frac{\log Z_G(\lam)}{n}
    \leq
    \frac{\log Z_{H_{r,d}}(\lam)}{rd}
    +O_{r,L}\left(\frac{(\log d)^2}{d^{1/(r+1)}}\right).
\]
In particular, taking \(\lam=1\) and by \eqref{eqn:Hrd_asympt},
\[
    \frac{\log i_r(G)}{n}
    \leq
    1+\frac{\log\left(1-2^{-(r-1)}\right)}{r}
    +O_r\left(\frac{(\log d)^2}{d^{1/(r+1)}}\right).
\]
\end{theorem}
In fact, this will follow by a more general result on a natural quotient of the hypergraph. Two vertices are called \emph{twins} if they belong to precisely the same hyperedges, and it is trivial to see that this is an equivalence relation. Contract every twin class to a single vertex, retaining one quotient edge for each original hyperedge, and call the resulting hypergraph \(Q(G)\) the \emph{twin quotient} of \(G\).  For a hypergraph \(Q\) and distinct vertices \(x,y\in V(Q)\), write \(d_Q(x,y)=|\{e\in E(Q):x,y\in e\}|\) for their codegree, and write
\[
    \Delta_2(Q)=\max_{\substack{x,y\in V(Q)\\x\neq y}}d_Q(x,y)
\]
for the maximum pair codegree.

\begin{theorem}\label{thm:container-general}
Fix \(r\geq3\) and \(L>0\).  Let \(G=G_d\) be an \(r\)-uniform, \(d\)-regular hypergraph on \(n\) vertices, where \(d\geq2\), and suppose that \(\Delta_2(Q(G))=o(d)\) as \(d\to\infty\).  Then, uniformly for \(0\leq\lam\leq L\),
\[
    \frac{\log Z_G(\lam)}{n}
    \leq
    \frac{\log Z_{H_{r,d}}(\lam)}{rd}
    +o(1).
\]
In particular, taking \(\lam=1\),
\[
    \frac{\log i_r(G)}{n}
    \leq
    1+\frac{\log\left(1-2^{-(r-1)}\right)}{r}
    +o(1).
\]
\end{theorem}

The condition is imposed on the quotient rather than on \(G\) itself because twins may have codegree \(d\), as they do in the conjectured extremal construction.  The proof uses the hypergraph container method \cite{balogh2015independent,saxton2015hypergraph}.

Observe that the twin quotient of every hypergraph with no cross-edges is linear. The quantitative estimate established in the proof of \Cref{thm:container-general} will therefore yield \Cref{thm:main_container}. This settles the asymptotic version of \Cref{conj:bbn_general} (ii) and (iii).

The local structure imposed by the absence of cross-edges allows us to obtain a substantially sharper error by the occupancy method \cite{davies2017independent}.

For \(0<\lam\leq1\), let
\[
    \Psi_{r,d}(\lam)=
    \begin{cases}
        \lam^{r-1}, & 0<\lam\leq d^{-1/(r-1)},\\
        \dfrac{1+\log(d\lam^{r-1})}{d}, & d^{-1/(r-1)}<\lam\leq1.
    \end{cases}
\]

\begin{theorem}\label{thm:cross-edge-free}
    Fix \(r\geq3\).  Let \(G\) be an \(r\)-uniform, \(d\)-regular hypergraph on \(n\) vertices with no cross-edges, where \(d\geq2\). Then for every \(0<\lam\leq1\),
    \[
        \frac{\log Z_G(\lam)}{n}
        \leq
        \frac{\log Z_{H_{r,d}}(\lam)}{rd}
        +O_r\big(\Psi_{r,d}(\lam)\big).
    \]
    In particular, taking \(\lam=1\),
    \[
        \frac{\log i_r(G)}{n}
        \leq
        1+\frac{\log\left(1-2^{-(r-1)}\right)}{r}
        +O_r\left(\frac{\log d}{d}\right).
    \]
\end{theorem}

In fact, the proofs of \Cref{thm:main_container,thm:container-general,thm:cross-edge-free} give the smaller main term
\[
    \frac{1}{r}\log P_r(\lam),
    \qquad
    P_r(\lam)=(1+\lam)\big((1+\lam)^{r-1}-\lam^{r-1}\big).
\]
The stated bounds follow since \(Z_{H_{r,d}}(\lam)\geq P_r(\lam)^d\) for every \(\lam\geq0\).

After conditioning outside the neighborhood of a random vertex, the no-cross-edge assumption forces the incident edges to form a sunflower whose core is the twin class of the root.  The residual local picture is controlled by the size of this class and the number of active petals.  An explicit occupancy certificate gives an upper bound for the occupancy fraction, which integrates to the fugacity estimate in \Cref{thm:cross-edge-free}.

We also record a stronger coefficient-wise form of \Cref{conj:bbn}.
\begin{conjecture}\label{conj:coef}
    Let \(G\) be an \(r\)-uniform, \(d\)-regular hypergraph on \(n\) vertices, where \(rd\mid n\). We write \(i_r^{(k)}(G)\) for the number of independent sets of \(G\) of size \(k\). Then
    \[
        i_r^{(k)}(G)\leq i_r^{(k)}(H_{r,d}^n),
    \]
    where \(H_{r,d}^n\) is the disjoint union of \(n/(rd)\) copies of \(H_{r,d}\).
\end{conjecture}
For \(r=2\), this is precisely a conjecture of Kahn \cite{kahn2001entropy}, proved for all large enough graphs in \cite{davies2021proof, davies2018tight}.  In \Cref{sec:four-coefficient} we verify this conjecture for cross-edge-free triple systems in the first nontrivial coefficient case \(k=4\).

We next turn from asymptotic results to exact results for $2$-regular hypergraphs. In this setting the hypergraph problem has a useful dual form: vertices of the hypergraph become edges of a loopless \(r\)-regular multigraph, and complements of independent sets become edge covers.  This lets us replace the original problem by a weighted edge-cover extremal problem.

\begin{theorem}\label{thm:odd-r}
\Cref{conj:bbn} holds for \(d=2\) and every odd \(r\geq3\).
\end{theorem}

The proof uses an entropy decomposition in the dual multigraph.  We partition the support graph into regions, condition on boundary edges, and use the Gibbs variational formula and Shearer's inequality to sum local free energies.  For odd \(r\), fractional matching theory reduces the required local estimates to single edges and odd cycles.

The first even case requires one additional ingredient, because a closed odd-cycle region can occur.

\begin{theorem}\label{thm:r4}
\Cref{conj:bbn} holds for \((r,d)=(4,2)\).
\end{theorem}

For \((r,d)=(4,2)\), closed chordless odd cycles are handled by a transfer-matrix calculation.  Closed odd cycles with chords are reduced to a finite family of rooted local configurations; the required local free-energy inequalities are then verified exactly.  The finite certificate is recorded in the appendix.

The proofs of \Cref{thm:odd-r,thm:r4} can be generalized from counting to the general partition function inequality for any $\lam\geq0$. Hence, in these cases, \Cref{conj:bbn_general} (ii) holds. We decided to restrict our presentation to $\lam=1$ for simplicity.

The remainder of the paper follows the three methods just described.  \Cref{sec:containers} proves the general asymptotic bound by the container method.  \Cref{sec:occupancy} sharpens the error in the cross-edge-free class by the occupancy method.  \Cref{sec:entropy} develops the edge-cover dual and proves \Cref{thm:odd-r}, while \Cref{sec:r4} proves \Cref{thm:r4} using the computational certificate given in the appendix.

\section{The container method: a general asymptotic bound}\label{sec:containers}

\subsection{Twin quotients}

For a vertex \(v\) of a hypergraph \(G\), write \(E_G(v)=\{e\in E(G):v\in e\}\). Thus two vertices are twins precisely when their sets \(E_G(v)\) are identical.  Let \(\mathcal B\) be the partition of \(V(G)\) into twin classes.  The twin quotient \(Q=Q(G)\) has vertex set \(\mathcal B\), and an edge \(\overline e=\{B\in\mathcal B:B\subseteq e\}\) for every \(e\in E(G)\). We attach to each quotient vertex \(B\) the positive integer weight \(c_B=|B|.\) We first record the structural properties of this quotient.

\begin{lemma}\label{lem:twin-quotient}
Let \(G\) be an \(r\)-uniform, \(d\)-regular hypergraph on \(n\) vertices, where \(d\geq2\), and let \(Q=Q(G)\).  Then the map \(e\mapsto\overline e\) is a bijection from \(E(G)\) to \(E(Q)\), every vertex of \(Q\) has degree \(d\), and \(\sum_{B\in\overline e}c_B=r\) for every \(\overline e\in E(Q)\).  Moreover, every edge of \(Q\) has size between \(2\) and \(r\), and \(|E(Q)|=nd/r\).
If \(G\) has no cross-edges, then \(Q\) is linear.
\end{lemma}

\begin{proof}
If \(B\in\mathcal B\) meets an edge \(e\), then every vertex of \(B\) belongs to \(e\), since the vertices of \(B\) have the same incident-edge set.  Thus every edge of \(G\) is a union of twin classes.  In particular, the quotient edge \(\overline e\) determines \(e\), so the map \(e\mapsto\overline e\) is injective and hence is a bijection by the definition of \(E(Q)\).

Every vertex in a class \(B\) belongs to the same \(d\) original edges, so \(B\) has degree \(d\) in \(Q\).  Since the classes in \(\overline e\) partition the original edge \(e\), we have \(\sum_{B\in\overline e}c_B=|e|=r\).  In particular, every quotient edge has size at most \(r\).  If some quotient edge consisted of a single class \(B\), then \(c_B=r\).  Every original edge incident to \(B\) would then be equal to \(B\), which is impossible in a simple hypergraph when \(d\geq2\).  Thus every quotient edge has size at least \(2\).  The identity \(|E(Q)|=nd/r\) follows from the corresponding identity for \(G\).

It remains to prove the last assertion.  Suppose that two distinct quotient edges \(\overline e,\overline f\) contain two distinct classes \(A,B\).  Since \(A\) and \(B\) are different twin classes, there is an edge \(g\) which contains one of them and not the other; without loss of generality, \(A\subseteq g\) and \(B\not\subseteq g\).  Choose \(a\in A\) and \(b\in B\).  Then the three distinct original edges \(e,f,g\) form a cross-edge, witnessed by \(b\in e\cap f\), \(a\in f\cap g\) and \(a\in g\cap e\), where \(b\notin e\cap f\cap g\).  This contradiction proves that \(Q\) is linear.
\end{proof}

The quotient also gives a convenient probabilistic representation of the partition function.  This representation does not require any codegree assumption.

\begin{lemma}\label{lem:full-blocks}
Let \(G\) and \(Q\) be as in \Cref{lem:twin-quotient}, fix \(\lam\geq0\), and put \(t=\lam/(1+\lam)\).  Let \(\mathbf F\subseteq V(Q)\) be the random set obtained by including the quotient vertices independently, with \(\P(B\in\mathbf F)=t^{c_B}\).  Then
\[
    Z_G(\lam)=(1+\lam)^n\P\big(\mathbf F\in\mathcal I(Q)\big).
\]
\end{lemma}

\begin{proof}
Choose a random subset \(\mathbf X\subseteq V(G)\) by including every original vertex independently with probability \(t\).  For every \(A\subseteq V(G)\),
\[
    \P(\mathbf X=A)
    =t^{|A|}(1-t)^{n-|A|}
    =\frac{\lam^{|A|}}{(1+\lam)^n}.
\]
Consequently, \(\P\big(\mathbf X\in\mathcal I(G)\big)=Z_G(\lam)/(1+\lam)^n\).

Let \(\mathbf F\) be the set of twin classes which are contained in \(\mathbf X\) in their entirety.  Since the classes are disjoint, their full-occupation events are independent, and a class \(B\) is full with probability \(t^{c_B}\).  Every original edge is the union of the classes in its quotient edge.  It follows that \(\mathbf X\) contains an edge of \(G\) if and only if \(\mathbf F\) contains an edge of \(Q\).  Thus \(\P\big(\mathbf X\in\mathcal I(G)\big)=\P\big(\mathbf F\in\mathcal I(Q)\big)\), which proves the lemma.
\end{proof}

\subsection{A bounded-rank container lemma}

We use the tight-container form of the hypergraph container theorem of Saxton and Thomason \cite[Corollary 3.6]{saxton2015hypergraph}.  Since the twin quotient need not be uniform, we first derive the bounded-rank form needed here.

\begin{lemma}\label{lem:bounded-rank-containers}
Fix \(r\geq3\).  Let \(Q\) be a \(d\)-regular hypergraph on \(N\) vertices, all of whose edges have sizes in \(\{2,\ldots,r\}\), and put \(\xi=\Delta_2(Q)/d\).
For \(C\subseteq V(Q)\), write \(e_Q(C)\) for the number of edges of \(Q\) contained in \(C\).
There are constants \(C_r>0\) and \(\xi_r>0\) such that, whenever \(0<\xi\leq\xi_r\), there is a family \(\mathcal C\) of subsets of \(V(Q)\) with the following properties:
\begin{enumerate}
    \item every independent set of \(Q\) is contained in some \(C\in\mathcal C\);
    \item \(e_Q(C)\leq C_r\xi^{1/(r+1)}|E(Q)|\) for every \(C\in\mathcal C\);
    \item \(\log|\mathcal C|\leq C_rN\xi^{1/(r+1)}\log^2(1/\xi)\).
\end{enumerate}
\end{lemma}

\begin{proof}
For \(2\leq k\leq r\), let \(Q_k\) be the \(k\)-uniform layer of \(Q\), and put \(m_k=|E(Q_k)|\), \(m=|E(Q)|\). Since \(Q\) is \(d\)-regular, \(dN=\sum_{k=2}^r km_k\), and hence
\begin{equation}\label{eq:bounded-rank-edge-count}
    \frac{dN}{r}\leq m\leq\frac{dN}{2}.
\end{equation}

Choose a constant \(A_r>0\), to be fixed below, and set \(\tau=A_r\xi^{1/(r+1)}\). By decreasing \(\xi_r\), we may assume that \(\tau<1/2\).  Call a layer \(Q_k\) significant if \(m_k\geq\tau m\).  The average degree \(\overline d_k\) of a significant layer satisfies, by \eqref{eq:bounded-rank-edge-count},
\[
    \overline d_k=\frac{km_k}{N}
    \geq\frac{2\tau m}{N}
    \geq\frac{2\tau d}{r}.
\]

For \(2\leq j\leq k\) and \(v\in V(Q)\), let \(d_k^{(j)}(v)\) be the maximum number of edges of \(Q_k\) containing a \(j\)-set which contains \(v\), as in the definition of the Saxton--Thomason codegree function.  Every such \(j\)-set contains a pair, and therefore \(d_k^{(j)}(v)\leq\Delta_2(Q)=\xi d\).  It follows directly from the definition of the codegree function that \(\delta(Q_k,\tau)\leq B_r\xi/\tau^r\) for a constant \(B_r\) depending only on \(r\).  Since \(\tau^{r+1}=A_r^{r+1}\xi\), we may choose \(A_r\) sufficiently large that \(\delta(Q_k,\tau)\leq\tau/(12k!)\) for every significant layer and every \(2\leq k\leq r\).

Apply the tight-container theorem to each significant \(Q_k\), with both of its parameters equal to \(\tau\).  We obtain a family \(\mathcal C_k\) such that every independent set in \(Q_k\) is contained in some member of \(\mathcal C_k\), every \(C\in\mathcal C_k\) satisfies \(e_{Q_k}(C)\leq\tau m_k\), and \(\log|\mathcal C_k|\leq C'_rN\tau\log^2(1/\tau)\).

Let \(\mathcal C\) consist of all intersections of one member of \(\mathcal C_k\) for each significant layer \(Q_k\).  An independent set of \(Q\) is independent in every layer, and hence is contained in one of these intersections.  There are at most \(r-1\) significant layers, so \(\log|\mathcal C|\leq C''_rN\tau\log^2(1/\tau)\).  For every \(C\in\mathcal C\), the significant layers contribute at most \(\tau m\) internal edges in total, while the insignificant layers contribute fewer than \((r-1)\tau m\).  Therefore \(e_Q(C)<r\tau m\).  Substituting \(\tau=A_r\xi^{1/(r+1)}\) and adjusting the constant \(C_r\) proves the lemma.
\end{proof}

\subsection{Proof of the general bound}

\begin{proof}[Proof of \Cref{thm:container-general}]
Let \(Q=Q(G)\), put \(\xi=\Delta_2(Q)/d\), and write \(N=|V(Q)|\) and \(m=|E(Q)|\).  By assumption, \(\xi=o(1)\), so \Cref{lem:bounded-rank-containers} applies for all sufficiently large \(d\).  By \Cref{lem:twin-quotient}, \(m=nd/r\) and \(N\leq n\).  Apply \Cref{lem:bounded-rank-containers} to \(Q\).  Thus there is a container family \(\mathcal C\) and quantities \(\eps_1=O_r(\xi^{1/(r+1)}\log^2(1/\xi))\) and \(\eps_2=O_r(\xi^{1/(r+1)})\) such that \(\log|\mathcal C|\leq\eps_1 n\) and \(e_Q(C)\leq\eps_2 m\) for every \(C\in\mathcal C\).

Fix \(C\in\mathcal C\), and put \(D=V(Q)\setminus C\).  Every edge of \(Q\) which is not internal to \(C\) meets \(D\).  Since every quotient vertex has degree \(d\), we have \(d|D|\geq m-e_Q(C)\geq(1-\eps_2)m\), and therefore
\begin{equation}\label{eq:container-complement-size}
    |D|\geq(1-\eps_2)\frac{n}{r}.
\end{equation}

Fix \(0\leq\lam\leq L\) and put \(t=\lam/(1+\lam)\).  Let \(\mathbf F\) be the random set from \Cref{lem:full-blocks}.  Since every quotient weight satisfies \(c_B\leq r-1\), we have \(1-t^{c_B}\leq1-t^{r-1}\).  It follows from \eqref{eq:container-complement-size} that
\[
    \P(\mathbf F\subseteq C)
    =\prod_{B\in D}(1-t^{c_B})
    \leq
    (1-t^{r-1})^{(1-\eps_2)n/r}.
\]
Every independent set of \(Q\) is contained in a member of \(\mathcal C\).  A union bound and \Cref{lem:full-blocks} now give
\[
    Z_G(\lam)
    \leq
    (1+\lam)^n|\mathcal C|
    (1-t^{r-1})^{(1-\eps_2)n/r}.
\]
Taking logarithms and dividing by \(n\), we obtain
\begin{align*}
    \frac{\log Z_G(\lam)}{n}
    &\leq
    \log(1+\lam)+\frac1r\log(1-t^{r-1})\\
    &\quad
    +\eps_1+\frac{\eps_2}{r}\log\frac1{1-t^{r-1}}.
\end{align*}
The last logarithm is bounded uniformly for \(0\leq\lam\leq L\).  Finally,
\[
    \log(1+\lam)+\frac1r\log(1-t^{r-1})
    =
    \frac1r\log\left((1+\lam)
    \big((1+\lam)^{r-1}-\lam^{r-1}\big)\right).
\]
The bounds on \(\eps_1\) and \(\eps_2\) give the quantitative error \(O_{r,L}(\xi^{1/(r+1)}\log^2(1/\xi))\), which is \(o(1)\) since \(\xi=o(1)\).  This proves the fugacity estimate.  The counting estimate follows by taking \(\lam=1\) and using \(Z_G(1)=i_r(G)\).
\end{proof}

\begin{proof}[Proof of \Cref{thm:main_container}]
By the final assertion of \Cref{lem:twin-quotient}, the twin quotient \(Q(G)\) is linear.  Since every edge of \(Q(G)\) has size at least two, \(\Delta_2(Q(G))=1\).  Applying the quantitative estimate from the proof of \Cref{thm:container-general} with \(\xi=1/d\) proves the fugacity estimate.  Taking \(\lam=1\) and using \(\log\big(2(2^{r-1}-1)\big)/r=1+\log(1-2^{-(r-1)})/r\) gives the counting estimate for all sufficiently large \(d\).  For the finitely many remaining degrees, the same bounds follow after enlarging the implicit constants, using the trivial estimate \(Z_G(\lam)\leq(1+\lam)^n\).
\end{proof}

\section{The occupancy method: a sharper general bound}\label{sec:occupancy}

Our goal is to prove \Cref{thm:cross-edge-free} via a bound on the occupancy fraction. The main estimate of this section is the following bound.

\begin{theorem}\label{thm:occupancy}
    Fix \(r\geq3\), and let \(G\) be an \(r\)-uniform, \(d\)-regular hypergraph on \(n\) vertices with no cross-edges, where \(d\geq2\). Then for every \(0<\lam\leq1\),
    \[
        \alpha_G(\lam)
        \leq
        \frac{\lam}{r}\left(
        \frac{1}{1+\lam}
        +(r-1)
        \frac{(1+\lam)^{r-2}-\lam^{r-2}}
        {(1+\lam)^{r-1}-\lam^{r-1}}
        \right)
        +O_r\left(\min\left\{
        \left(\frac{\lam}{1+\lam}\right)^{r-1},\frac1d
        \right\}\right).
    \]
\end{theorem}
The implicit constant depends only on \(r\), uniformly in
\(G,d\) and \(\lam\).

Integrating this estimate gives the fugacity bound in \Cref{thm:cross-edge-free}.  We will also use this bound to record a coefficient consequence.
\begin{corollary}[Coefficient bound]\label{cor:coefficient_ld}
    Let \(i_r^{(k)}(G)\) be the number of independent sets of size \(k\), and use the convention \(\log 0=-\infty\). Under the same hypotheses, there is a constant \(C_r\) such that for every \(0\leq k\leq n\),
    \[
        \frac{\log i_r^{(k)}(G)}{n}
        \leq
        \inf_{0<\lam\leq1}
        \left\{
        \frac{1}{r}\log\left((1+\lam)
        \big((1+\lam)^{r-1}-\lam^{r-1}\big)\right)
        -\frac{k}{n}\log \lam
        +C_r\Psi_{r,d}(\lam)
        \right\}.
    \]
\end{corollary}

Before turning to the local occupancy argument, we prove the promised small coefficient case of \Cref{conj:coef}.

\begin{proposition}\label{prop:four-coefficient}
    Let \(G\) be a \(3\)-uniform, \(d\)-regular hypergraph on \(n\) vertices with no cross-edges, where \(d\geq1\) and \(3d\mid n\). Then \(i_3^{(4)}(G)\leq i_3^{(4)}(H_{3,d}^n)\).
\end{proposition}

\subsection{A coefficient warm-up}\label{sec:four-coefficient}

For a hypergraph \(G=(V,E)\), its incidence graph is defined as a bipartite graph \(\Gamma_G\), with parts \(V\) and \(E\), and an edge \((v,e)\in V\times E\) if \(v\in e\) in \(G\). We are now ready to prove \Cref{prop:four-coefficient}.
\begin{proof}
Let \(G\) be a \(3\)-uniform \(d\)-regular hypergraph on \(n\) vertices with no cross-edges. We first express \(i_3^{(4)}(G)\) in terms of pairs of edges which share two vertices.

Starting with all \(4\)-sets, subtract those containing a prescribed edge. There are \(nd/3\) edges and, after choosing one of them, \(n-3\) choices for the fourth vertex. This gives \(\binom{n}{4}-nd(n-3)/3\).
The only overcount comes from a \(4\)-set containing two distinct hyperedges. Since the hypergraph is \(3\)-uniform, this happens exactly when the two hyperedges have union of size \(4\), equivalently when they share two vertices. Moreover, no \(4\)-set can contain three hyperedges: for three distinct \(3\)-subsets of a \(4\)-set, each pairwise intersection contains an element omitted by the third triple, and choosing these three elements gives witnesses for a cross-edge. Hence
\[
    i_3^{(4)}(G)
    =
    \binom{n}{4}-\frac{nd(n-3)}{3}
    +\#\{ \{e,f\}: e\neq f,\ |e\cap f|=2\}.
\]

Let \(\Gamma_G\) be the incidence graph of \(G\). A pair of hyperedges \(e,f\) with \(|e\cap f|=2\), say \(e\cap f=\{a,b\}\), gives the \(4\)-cycle \(a-e-b-f-a\) in \(\Gamma_G\). Conversely, every \(4\)-cycle in \(\Gamma_G\) consists of two vertices of \(G\) and two hyperedges containing both of them, and therefore gives such a pair of hyperedges. Thus the last term is exactly the number of \(C_4\)'s in \(\Gamma_G\).

It remains to prove that \(\# C_4(\Gamma_G)\leq nd(d-1)/6\).
Recall that \(d_G(x,y)\) denotes the codegree of a pair of distinct vertices \(x,y\in V(G)\), so that \(\# C_4(\Gamma_G)=\sum_{\{x,y\}}\binom{d_G(x,y)}{2}\).
We claim that, for every hyperedge \(e=\{a,b,c\}\), at most one of the three pairs
\(\{a,b\},\{a,c\},\{b,c\}\) has codegree at least \(2\). Indeed, if for instance
there were hyperedges \(f\neq e\) containing \(\{a,b\}\) and \(g\neq e\) containing
\(\{a,c\}\), then \(e,f,g\) would be distinct and would form a cross-edge, witnessed
by \(b\in e\cap f\), \(a\in f\cap g\) and \(c\in g\cap e\).
Thus each hyperedge contains at most one pair \(\{x,y\}\) with \(d_G(x,y)\geq2\), and so \(\sum_{\{x,y\}:d_G(x,y)\geq2}d_G(x,y)\leq |E(G)|=nd/3\).
Since \(d_G(x,y)\leq d\), we obtain
\[
    \# C_4(\Gamma_G)
    =
    \sum_{\{x,y\}:d_G(x,y)\geq2}\binom{d_G(x,y)}{2}
    \leq
    \frac{d-1}{2}\sum_{\{x,y\}:d_G(x,y)\geq2}d_G(x,y)
    \leq
    \frac{nd(d-1)}{6}.
\]

For one copy of \(H_{3,d}\), the only pairs of hyperedges meeting in two vertices are obtained by fixing one of the \(d\) non-marked pairs and choosing two of the \(d\) marked vertices. Thus it has \(d\binom{d}{2}=d^2(d-1)/2\) such pairs. Moreover, a \(4\)-set containing two hyperedges in \(H_{3,d}\) consists of one non-marked pair together with two marked vertices, and therefore contains exactly those two hyperedges. The disjoint union \(H_{3,d}^n\) has \(n/(3d)\) copies, hence exactly \(nd(d-1)/6\) such pairs. Therefore
\[
    i_3^{(4)}(H_{3,d}^n)
    =
    \binom{n}{4}-\frac{nd(n-3)}{3}+\frac{nd(d-1)}{6}.
\]
Substituting the bound above into the displayed formula for \(i_3^{(4)}(G)\) gives \(i_3^{(4)}(G)\leq i_3^{(4)}(H_{3,d}^n)\), as required.

\end{proof}

\subsection{The local linear program}
\noindent Let \(G\) be as in \Cref{thm:occupancy}.  Fix
\(0<\lam\leq1\), draw \(\mathbf I\) from the hard-core model, and
independently choose a uniformly random vertex \(\mathbf v\).  Throughout
this section put \(t=\lam/(1+\lam)\) and \(s=r-1\),
and, for \(1\leq j\leq s\), put
\[
    y_j=\frac{t-t^j}{1-t^j},
    \qquad
    \delta_j=t-y_j=\frac{t^j(1-t)}{1-t^j},
    \qquad
    \chi_j=dt^j.
\]
Thus \(y_j\) is the occupation probability of a specified vertex in a
\(j\)-set conditioned not to occupy the whole set.

We first record the local structure forced by the absence of cross-edges.

\begin{lemma}\label{lem:general-local-structure}
Let \(v\) belong to a twin class \(B\) of size
\(1\leq c\leq r-1\), and put \(q=r-c\).
Then the \(d\) edges containing \(v\) have the form
\(B\cup P_1,\ldots,B\cup P_d\) with \(|P_i|=q\),
where the petals \(P_i\) are pairwise disjoint.  Moreover, an edge not
containing \(v\) meets at most one of the petals.
\end{lemma}

\begin{proof}
Suppose two edges containing \(v\) also contain a vertex \(x\neq v\).
If a third edge containing \(v\) did not contain \(x\), these three edges
would form a cross-edge, witnessed by \(x,v,v\).  Hence \(x\) belongs to
every edge containing \(v\).  Since \(x\) and \(v\) both have degree \(d\),
they are twins.  It follows that the common intersection of the edges
through \(v\) is precisely \(B\), and their complements \(P_i\) are
pairwise disjoint.

If an outside edge \(h\) met two distinct petals \(P_i\) and \(P_j\), then
\(B\cup P_i,h,B\cup P_j\) would form a cross-edge, witnessed by one point
from each of \(h\cap P_i,h\cap P_j\), together with \(v\).  This proves the
last assertion.
\end{proof}

In particular, the neighborhood of \(v\) is
\(N_v=(B\setminus\{v\})\sqcup P_1\sqcup\cdots\sqcup P_d\), of size
\(|N_v|=c-1+dq\).
Fix a vertex \(v\in B\), put \(U_B=B\sqcup P_1\sqcup\cdots\sqcup P_d\) and
\(J=\mathbf I\setminus U_B\), and expose \(J\).  By \Cref{lem:general-local-structure}, an edge avoiding
\(v\) is disjoint from \(B\) and meets at most one petal.  Hence the
residual constraints factor over the petals.  For each \(i\), let
\(\mathcal F_i(J)\subseteq2^{P_i}\) consist of the sets \(A\subseteq P_i\)
such that \(f\cap P_i\nsubseteq A\)
for every edge \(f\) avoiding \(v\) with
\(f\setminus U_B\subseteq J\) and \(f\cap P_i\neq\varnothing\).
This is a nonempty downset.  We call \(P_i\) \emph{active} if
\(\mathcal F_i(J)=2^{P_i}\).  On an active petal the original edge
\(B\cup P_i\) remains as a constraint.  If the petal is inactive, then
\(\mathcal F_i(J)\) is proper, the full petal is already forbidden, and
the original edge is redundant.

We define the relaxed local configuration \(\sigma_v\) to be the pair
\((c,k)\), where \(c=|B|\) and \(k\) is the number of active petals.
Thus, writing \(p_{c,k}=\P(\sigma_{\mathbf v}=(c,k))\), we have
\[
    \alpha_G(\lam)
    =
    \sum_{c=1}^{r-1}\sum_{k=0}^d
    p_{c,k}\P(\mathbf v\in\mathbf I\mid \sigma_{\mathbf v}=(c,k)).
\]
The configuration records no further information about an inactive
petal; the local LYM estimate below is the relaxation which makes this
coarsening possible.

We need the following elementary bound for an inactive petal.

\begin{lemma}\label{lem:proper-downset}
Let \(\varnothing\neq\mathcal F\subsetneq2^{[q]}\) be a proper downset, and choose
\(\mathbf A\in\mathcal F\) with probability proportional to
\(\lam^{|\mathbf A|}\).  Then \(\E|\mathbf A|\leq qy_q\).
\end{lemma}

\begin{proof}
For \(0\leq j\leq q\), write
\[
    f_j=\frac{|\mathcal F\cap\binom{[q]}j|}{\binom qj}.
\]
The local LYM inequality gives \(f_0\geq f_1\geq\cdots\geq f_q=0\).
Let \(\nu\) be the probability distribution on
\(\{0,\ldots,q-1\}\) proportional to \(\binom qj\lam^j\).
The rank distribution on \(\mathcal F\) is obtained by reweighting
\(\nu\) by the nonincreasing sequence \(f_j\), and hence has no larger
mean.  Therefore
\[
    \E|\mathbf A|
    \leq
    \lam\frac{\mathrm d}{\mathrm d\lam}
    \ln\big((1+\lam)^q-\lam^q\big)
    =qy_q.
\]
\end{proof}

Suppose now that exactly \(k\) of the \(d\) petals are active.  The
probability that the whole block \(B\) is occupied is
\[
    w_{c,k}
    =
    \frac{t^c(1-t^q)^k}
    {1-t^c+t^c(1-t^q)^k}.
\]
Conditional on \(B\) not being full, a specified vertex of \(B\) has
occupation probability \(y_c\).  Thus
\[
    \alpha_{c,k}^v
    =\P(v\in\mathbf I\mid \sigma_v=(c,k))
    =y_c+(1-y_c)w_{c,k}.
\]
An active petal has expected occupation
\(q\big((1-w_{c,k})t+w_{c,k}y_q\big)=q(t-w_{c,k}\delta_q)\),
whereas \Cref{lem:proper-downset} bounds an inactive petal by \(qy_q\).
Consequently,
\[
    \frac1{dq}\sum_{i=1}^d
    \E\bigl(|\mathbf I\cap P_i|\mid \sigma_v=(c,k)\bigr)
    \leq
    \alpha_{c,k}^N,
\]
where
\[
    \alpha_{c,k}^N
    =
    \frac{k}{d}(t-w_{c,k}\delta_q)
    +\left(1-\frac{k}{d}\right)y_q.
\]
For a twin class \(B\), write \(c_B=|B|\), \(q_B=r-c_B\), and let
\(\alpha_B^v\) be the occupation probability of any specified vertex of
\(B\).  Define
\[
    \alpha_B^N=
    \frac{1}{dq_B}
    \sum_{e\supseteq B}
    \sum_{\substack{B'\in\mathcal B,\ B'\subseteq e\\B'\neq B}}
    c_{B'}\alpha_{B'}^v,
\]
where the inner sum is over twin classes contained in \(e\).  We shall
use the following transport identity:
\begin{align*}
    \sum_Bc_Bq_B\alpha_B^N
    &=
    \frac1d\sum_{e\in E(G)}
    \sum_{B\subseteq e}c_B
    \sum_{\substack{B'\subseteq e\\B'\neq B}}c_{B'}\alpha_{B'}^v\\
    &=
    \frac1d\sum_{e\in E(G)}
    \sum_{B'\subseteq e}c_{B'}q_{B'}\alpha_{B'}^v
    =
    \sum_Bc_Bq_B\alpha_B^v.
\end{align*}

For \(1\leq c\leq r-1\) and \(0\leq k\leq d\), let
\[
    u_{c,k}
    =
    p_{c,k}\,
    \frac1{d(r-c)}
    \sum_{i=1}^d
    \E\bigl(|\mathbf I\cap P_i|
    \mid \sigma_{\mathbf v}=(c,k)\bigr).
\]
If \(p_{c,k}=0\), we use the convention \(u_{c,k}=0\).
The local LYM estimate gives the linear inequalities
\(0\leq u_{c,k}\leq p_{c,k}\alpha_{c,k}^N\).
The transport identity above is equivalently
\[
    \sum_{c=1}^{r-1}\sum_{k=0}^d
    (r-c)\bigl(p_{c,k}\alpha_{c,k}^v-u_{c,k}\bigr)=0.
\]

Let \(\Pi_r\) be the set of integer partitions of \(r\) whose parts are
smaller than \(r\), and let \(m_c(\pi)\) be the multiplicity of the part
\(c\) in \(\pi\).  Let \(z_\pi\) be the proportion of quotient edges
whose class-size partition is \(\pi\).  Double-counting incidences of
classes of size \(c\) gives
\[
    \sum_{k=0}^d p_{c,k}
    =
    \frac{c}{r}\sum_{\pi\in\Pi_r}m_c(\pi)z_\pi.
\]
Taking into account that the \(p_{c,k}\)'s and \(z_\pi\)'s are
probability distributions, we obtain the following LP:
\begin{align*}
    \alpha_{\mathrm{LP}}(\lam)
    =\max\quad&
    \sum_{c=1}^{r-1}\sum_{k=0}^d
    p_{c,k}\alpha_{c,k}^v
    \qquad\text{s.t.}\\
    &\sum_{c=1}^{r-1}\sum_{k=0}^d
    (r-c)\bigl(p_{c,k}\alpha_{c,k}^v-u_{c,k}\bigr)=0,\\
    &\sum_{k=0}^d p_{c,k}
    =
    \frac{c}{r}\sum_{\pi\in\Pi_r}m_c(\pi)z_\pi
    \qquad(1\leq c\leq r-1),\\
    &\sum_{\pi\in\Pi_r}z_\pi=1,\\
    &0\leq u_{c,k}\leq p_{c,k}\alpha_{c,k}^N,
    \qquad p_{c,k}\geq0
    \qquad(1\leq c\leq r-1,\ 0\leq k\leq d),\\
    &z_\pi\geq0\qquad(\pi\in\Pi_r).
\end{align*}
The class-size constraints and \(\sum_\pi z_\pi=1\) imply
\(\sum_{c,k}p_{c,k}=1\).
For fixed \(r\) and \(d\), this is a finite linear program, and the
distribution supplied by \(G\) is a feasible primal point whose
objective value is \(\alpha_G(\lam)\).  Consequently any
feasible dual solution gives an upper bound on the occupancy fraction.
The full dual initially has one nonnegative multiplier for each upper
bound \(u_{c,k}\leq p_{c,k}\alpha_{c,k}^N\).  These multipliers may be
eliminated.  Indeed, no optimum has \(\Lam_T>0\), since replacing a
positive \(\Lam_T\) by \(0\) only weakens the local constraints.  For
\(\Lam_T\leq0\), the optimal multiplier of the upper bound is
\(-\Lam_Tq\).  The dual LP therefore has variables
\(\Lam_P,\Lam_T,\Lam_1,\ldots,\Lam_{r-1}\):
\begin{align*}
    \min\quad&\Lam_P
    \qquad\text{s.t.}\\
    &\Lam_c+\Lam_Tq
    \bigl(\alpha_{c,k}^N-\alpha_{c,k}^v\bigr)
    \geq\alpha_{c,k}^v
    \qquad(c+q=r,\ 0\leq k\leq d),\\
    &\frac1r\sum_{c=1}^{r-1}
    c\,m_c(\pi)\Lam_c\leq\Lam_P
    \qquad(\pi\in\Pi_r),\\
    &\Lam_T\leq0.
\end{align*}
Here \(\Lam_T\) is the dual variable corresponding to the transport
constraint, \(\Lam_c\) corresponds to the class-size constraint for
\(c\), and \(\Lam_P\) corresponds to the normalization of \(z\).

\subsection{Verification of the dual certificate}
We first treat the range \(\chi_s\geq1\).  The candidate extremal
configuration \(H_{r,d}\) has class-size partition \(1+s\).  We choose
the transport dual variable by making the two endpoint constraints with
\(c=1\) tight.  Put
\[
    W=w_{1,d},
    \qquad
    a_s=\frac{t-W}{t-W+\delta_s(1-W)},
    \qquad
    \Lam_T=-\frac{a_s}{s}.
\]
For every \(1\leq q\leq s\), put \(a_q=-\Lam_Tq=a_sq/s\).
For \(c+q=r\) and \(0\leq k\leq d\), define
\begin{equation}\label{eq:lck}
    L_{c,k}
    =
    \alpha_{c,k}^v
    -\Lam_Tq\bigl(\alpha_{c,k}^N-\alpha_{c,k}^v\bigr)
    =
    (1-a_q)\alpha_{c,k}^v
    +a_q\alpha_{c,k}^N,
\end{equation}
and put \(\Lam_c=\max_{0\leq k\leq d}L_{c,k}\).
Finally, set
\[
    \Lam_P=
    \frac1r\max_{\pi\in\Pi_r}
    \sum_{c=1}^{r-1}c\,m_c(\pi)\Lam_c.
\]
These choices satisfy all dual constraints.  Hence
\begin{equation}\label{eq:global-certificate}
    \alpha_G(\lam)
    \leq\alpha_{\mathrm{LP}}(\lam)
    \leq\Lam_P.
\end{equation}

\subsubsection{Local constraints}
It remains to control the local maximum defining \(\Lam_c\).  The following
endpoint reduction replaces the full family of local constraints by two
constraints for each class size.

\begin{lemma}\label{lem:endpoint-local}
For every \(1\leq c\leq r-1\), we have \(\Lam_c=\max\{L_{c,0},L_{c,d}\}\).
Moreover,
\begin{align}
    L_{c,0}&=t-a_q\delta_q,\label{eq:l-c-zero}\\
    L_{c,d}
    &=t-(1-a_q)\delta_c
    +w_{c,d}\big((1-a_q)(1-y_c)-a_q\delta_q\big),\label{eq:l-c-d}
\end{align}
the coefficient of \(w_{c,d}\) in \eqref{eq:l-c-d} is nonnegative, and
\(L_{1,0}=L_{1,d}\).
\end{lemma}

\begin{proof}
Put \(x=1-t^q\), \(\zeta_k=t^cx^k/(1-t^c)\) and
\(w_{c,k}=\zeta_k/(1+\zeta_k)\).
Up to a term independent of \(k\), the expression in \eqref{eq:lck} is
\(\beta w_{c,k}+\gamma k(1-w_{c,k})\), where \(\beta=(1-a_q)(1-y_c)\) and
\(\gamma=a_q\delta_q/d\).
A direct subtraction shows that the sign of
\(L_{c,k+1}-L_{c,k}\) is the sign of
\(M_k=\gamma/\zeta_k+\gamma+(1-x)(\gamma k-\beta)\).
Furthermore, \(M_{k+1}-M_k=\gamma(1/x-1)/\zeta_k+\gamma(1-x)>0\).
Thus the successive differences change sign at most once, from negative
to positive, and the maximum is attained at an endpoint.

The displayed endpoint formulas follow directly from \eqref{eq:lck}.
We verify the sign of the remaining coefficient.  Since \(0<a_s<1\),
we have \(a_q\leq q/s\).  Also \(1-y_c\geq1-t\).  If \(c\geq2\), then
\begin{align*}
    &(1-a_q)(1-y_c)-a_q\delta_q\\
    &\qquad\geq
    \frac1s\big((c-1)(1-t)-q\delta_q\big)\\
    &\qquad=
    \frac{1-t}{s(1-t^q)}\big(c-1-st^q\big)\geq0,
\end{align*}
because \(s=c+q-1\leq(c-1)2^q\) and \(t\leq1/2\).  If \(c=1\), then
\[
    1-a_s
    =
    \frac{\delta_s(1-W)}{t-W+\delta_s(1-W)}
    \geq a_s\delta_s.
\]
Finally, the definition of \(a_s\) is precisely the equality
\((1-a_s)(t-W)=a_s\delta_s(1-W)\),
which is equivalent to \(L_{1,0}=L_{1,d}\).
\end{proof}
\subsubsection{Partition constraints}
The finite certificate above converges to a particularly simple
expression.  For \(c+q=r\), put
\[
    a_q^*
    =
    \frac{q}{s}\frac{t}{t+\delta_s},
    \qquad
    D_c=
    \min\{a_q^*\delta_q,(1-a_q^*)\delta_c\}.
\]
The limiting endpoint value is \(t-D_c\).  The following inequality shows
that the partition \(1+s\), corresponding to \(H_{r,d}\), gives the
largest edge contribution.

\begin{lemma}\label{lem:limiting-partition}
For every partition \(\pi\) of \(r\) into parts smaller than \(r\), we have
\(\sum_{c\in\pi}cD_c\geq s\delta_s\).
Equality is attained by the partition \(1+s\).
\end{lemma}

\begin{proof}
Put \(\rho=\delta_s/t\) and \(\theta=1/(1+\rho)=t/(t+\delta_s)\).
Since \(t\leq1/2\), we have \(s\delta_s\leq t\), and hence
\(\theta\geq s/(s+1)\).  Write \(f_c=cD_c=\min\{g_c,h_c\}\),
where, as before, \(q=r-c\),
\[
    g_c=\frac{cq\theta}{s}\delta_q,
    \qquad
    h_c=\frac{c(c-1+s\rho)\theta}{s}\delta_c.
\]
At the endpoints, \(f_1=\theta\delta_s\) and \(f_s=(s-\theta)\delta_s\),
so \(f_1+f_s=s\delta_s\).
Indeed, the two branches defining \(D_1\) agree, while
\(\rho\leq1/s\) shows that the second branch defines \(D_s\).

Let \(\ell_c\) be the chord joining these two endpoint values:
\[
    \ell_c
    =
    \delta_s\left(
    \theta+\frac{c-1}{s-1}(s-2\theta)
    \right).
\]
Since \(\theta\geq s/(s+1)\), we have \(\ell_c\leq c\theta\delta_s\leq c\delta_s\).
Also,
\[
    \frac{\delta_q}{\delta_s}
    =
    \frac{\sum_{i=1}^st^{-i}}{\sum_{i=1}^qt^{-i}}
    \geq\frac{s}{q},
\]
and therefore \(g_c\geq c\theta\delta_s\geq\ell_c\).

For the second branch, \((c-1+s\rho)/(1+\rho)\geq c-1\), so
\(h_c\geq c(c-1)\delta_c/s\).
If \(s\geq4\) and \(2\leq c\leq s-1\), then
\[
    \frac{\delta_c}{\delta_s}
    =
    t^{c-s}\frac{1-t^s}{1-t^c}
    \geq2^{s-c},
    \qquad
    (c-1)2^{s-c}\geq s.
\]
Thus \(h_c\geq c\delta_s\geq\ell_c\).  The only remaining interior
case is \(s=3,c=2\).  Here \(\ell_2=3\delta_3/2\) and \(h_2\geq2\delta_2/3\),
and \(\delta_2/\delta_3=1/t+t/(1+t)\geq7/3>9/4\).
There is no interior case when \(s=2\).  We have therefore proved
\(f_c\geq\ell_c\) for every \(c\).

Suppose that \(\pi\) has \(m\geq2\) parts.  Since their sum is \(s+1\),
\begin{align*}
    \sum_{c\in\pi}f_c
    &\geq
    mf_1+\frac{f_s-f_1}{s-1}(s+1-m)\\
    &=
    f_1+f_s
    +(m-2)\left(
    f_1-\frac{f_s-f_1}{s-1}
    \right)\\
    &\geq f_1+f_s=s\delta_s,
\end{align*}
where the last inequality is again equivalent to
\(\theta\geq s/(s+1)\).
\end{proof}
\subsection{Proof of the occupancy and counting bounds}
By the preceding section, the candidate dual point is feasible.  We now
compare its finite endpoint values in
\Cref{lem:endpoint-local} with the limiting quantities in
\Cref{lem:limiting-partition}.  Continue to assume that \(\chi_s\geq1\).
Since
\[
    W=
    \frac{t(1-t^s)^d}{1-t+t(1-t^s)^d},
\]
we have \(W\leq2te^{-\chi_s}\) and \(t-W\geq(1-2/e)t\).
Writing \(a_s^*=t/(t+\delta_s)\),
a direct subtraction gives
\[
    a_s^*-a_s
    =
    \frac{W\delta_s(1-t)}
    {(t+\delta_s)(t-W+\delta_s(1-W))}
    \leq
    C_r\frac{\delta_s}{t}e^{-\chi_s}.
\]
It follows from \eqref{eq:l-c-zero} that the finite \(k=0\) endpoint
exceeds its limiting value by at most
\[
    (a_q^*-a_q)\delta_q
    \leq
    C_r\frac{\delta_s\delta_q}{t}e^{-\chi_s}
    \leq C_rt^se^{-\chi_s}
    \leq\frac{C_r}{d}.
\]

For the other endpoint, \eqref{eq:l-c-d} and the nonnegativity proved in
\Cref{lem:endpoint-local} give
\[
    L_{c,d}
    -\big(t-(1-a_q^*)\delta_c\big)
    \leq w_{c,d}(1-a_q).
\]
If \(c=1\), then \(1-a_s\leq C_r\delta_s/t\), and hence
\(W(1-a_s)\leq C_r\delta_se^{-\chi_s}\leq C_r/d\).
Suppose \(c\geq2\).  Then \(w_{c,d}\leq2t^ce^{-\chi_q}\).
If \(c\geq q\), then \(\chi_ce^{-\chi_q}=\chi_qt^{c-q}e^{-\chi_q}\leq\chi_qe^{-\chi_q}\).
If \(c<q\), the assumption \(\chi_s\geq1\) gives \(\chi_qt^{c-1}\geq1\), and
\(\chi_ce^{-\chi_q}=\chi_qt^{-(q-c)}e^{-\chi_q}\leq\chi_q^{1+(q-c)/(c-1)}e^{-\chi_q}\).
Both expressions are bounded by a constant depending only on \(r\).
We conclude that \(\Lam_c\leq t-D_c+C_r/d\) uniformly over
\(1\leq c\leq r-1\).
Combining this with
\eqref{eq:global-certificate} and \Cref{lem:limiting-partition} yields
\begin{equation}\label{eq:occupancy-large}
    \alpha_G(\lam)
    \leq
    t-\frac{s\delta_s}{r}+\frac{C_r}{d}
    \qquad(\chi_s\geq1).
\end{equation}

It remains to consider \(\chi_s<1\).  In this range one may use the trivial
dual solution
\[
    \Lam_T=0,\qquad
    \Lam_c=\Lam_P=t.
\]
Indeed, conditional on the full configuration off a vertex, its
occupation probability is either \(0\) or \(t\).  Averaging over the
configurations represented by \((c,k)\) gives
\(\alpha_{c,k}^v\leq t\), so all dual constraints hold.  Therefore
\begin{equation}\label{eq:occupancy-small}
    \alpha_G(\lam)
    \leq
    t
    =
    t-\frac{s\delta_s}{r}
    +O_r(t^s)
    \qquad(\chi_s<1).
\end{equation}
Since \(t-s\delta_s/r=(t+sy_s)/r\),
\eqref{eq:occupancy-large} and \eqref{eq:occupancy-small} prove
\[
    \alpha_G(\lam)
    \leq
    \frac{t+sy_s}{r}
    +O_r\left(\min\left\{t^s,\frac1d\right\}\right).
\]
In terms of \(\lam\), the main term is
\[
    \frac{\lam}{r}\left(
    \frac{1}{1+\lam}
    +(r-1)\frac{(1+\lam)^{r-2}-\lam^{r-2}}
    {(1+\lam)^{r-1}-\lam^{r-1}}
    \right).
\]
This proves \Cref{thm:occupancy}.
The main term in \Cref{thm:occupancy} satisfies
\[
    \lam\frac{\mathrm d}{\mathrm d\lam}
    \left[
    \frac1r\ln\left((1+\lam)
    \big((1+\lam)^s-\lam^s\big)\right)
    \right]
    =
    \frac{t+sy_s}{r}.
\]
Consequently, for \(0<\mu\leq1\), the occupancy identity gives
\begin{align*}
    \frac{\log Z_G(\mu)}{n}
    &=
    \frac1{\ln2}\int_0^\mu
    \frac{\alpha_G(\lam)}{\lam}\,\mathrm d\lam\\
    &\leq
    \frac1r\log\left((1+\mu)
    \big((1+\mu)^s-\mu^s\big)\right) +
    O_r\left(
    \int_0^\mu
    \min\left\{
    \left(\frac{\lam}{1+\lam}\right)^s,\frac1d
    \right\}\frac{\mathrm d\lam}{\lam}
    \right).
\end{align*}
If \(\mu\leq d^{-1/s}\), the error integral is \(O_r(\mu^s)\).
Otherwise, splitting at \(d^{-1/s}\) and using
\(t\leq\lam\) gives \(O_r\big(1/d+\log(d\mu^s)/d\big)\).
This is \(O_r(\Psi_{r,d}(\mu))\), and proves the fugacity estimate in
\Cref{thm:cross-edge-free}.  Taking \(\mu=1\) and using
\(Z_G(1)=i_r(G)\) gives
\[
    \frac{\log i_r(G)}{n}
    \leq
    1+\frac{\log(1-2^{-(r-1)})}{r}
    +O_r\left(\frac{\log d}{d}\right).
\]
This proves \Cref{thm:cross-edge-free}.
\begin{proof}[Proof of \Cref{cor:coefficient_ld}]
If \(i_r^{(k)}(G)=0\), there is nothing to prove. Otherwise, \(i_r^{(k)}(G)\lam^k\leq Z_G(\lam)\) for every \(0<\lam\leq1\).
Taking base-\(2\) logarithms and applying \Cref{thm:cross-edge-free}, with \(C_r\) chosen larger than the implicit constant there, gives the claimed bound for each \(\lam\). Taking the infimum over \(0<\lam\leq1\) completes the proof.
\end{proof}

The container and occupancy arguments above concern asymptotic bounds under a pair-codegree or cross-edge restriction.  We now turn to exact results without either restriction when every vertex belongs to precisely two hyperedges.

\section{The entropy method: the degree-two problem for odd uniformity}\label{sec:entropy}

\subsection{The edge-cover dual}
\label{sec:duality}

Let \(F\) be a loopless multigraph.  An edge cover of \(F\) is a subset of its edge multiset incident to every vertex.  We write \(\ec(F)\) for the number of edge covers of \(F\).

\begin{lemma}
\label{lem:duality}
\Cref{conj:bbn} for \(d=2\) and uniformity \(r\) is equivalent to the following assertion.  If \(F\) is a loopless \(r\)-regular multigraph on \(N\) vertices with edge multiplicity at most \(r-1\), then
\begin{equation}\label{eq:main-edge-cover-bound}
  \ec(F)\leq T_r^{N/4},
\end{equation}
where \(T_r=2^{2(r-1)}+3(2^{r-1}-1)^2\).
\end{lemma}

\begin{proof}
Let \(G\) be a simple \(2\)-regular \(r\)-uniform hypergraph.  Define a multigraph \(F\) as follows.  The vertices of \(F\) are the edges of \(G\).  A vertex \(x\in V(G)\), lying in the two hyperedges \(e\) and \(f\), gives an edge \(ef\) in \(F\).  Since \(G\) is \(r\)-uniform, \(F\) is \(r\)-regular.  Since \(G\) is \(2\)-regular, every hypergraph vertex gives exactly one graph edge.  Simplicity of \(G\) implies that the multiplicity of every edge of \(F\) is at most \(r-1\).

Conversely, every loopless \(r\)-regular multigraph with maximum edge multiplicity at most \(r-1\) arises in this way.  Take the edges of \(F\) as the vertices of a hypergraph, and for each vertex \(v\in V(F)\), take one hyperedge consisting of the \(r\) graph edges incident to \(v\).  Looplessness gives \(2\)-regularity, regularity gives \(r\)-uniformity.  If two resulting hyperedges were identical, all \(r\) edges incident to one endpoint would be parallel to the other endpoint, contradicting the multiplicity bound.

Under this correspondence, an independent set in \(G\) is the complement of an edge cover in \(F\), so \(i_r(G)=\ec(F)\).
Moreover, if \(|V(F)|=N\), then \(|V(G)|=|E(F)|=rN/2\), so the exponent \(n/(2r)\) in \Cref{conj:bbn} becomes \(N/4\).

Let \(F_r\) denote the dual of \(H_{r,2}\).  It is the four-vertex \(r\)-regular multigraph obtained from a cycle by giving two opposite edges multiplicity \(r-1\) and the other two opposite edges multiplicity \(1\).  A direct count gives \(\ec(F_r)=2^{2(r-1)}+3(2^{r-1}-1)^2=T_r\).
Indeed, writing \(h=2^{r-1}-1\), the weighted edge-cover count on the underlying four-cycle is \(1+2h+4h^2=(h+1)^2+3h^2\).
\end{proof}

It remains to prove the edge-cover inequality in \Cref{lem:duality}.  Let \(F\) be a loopless \(r\)-regular multigraph with multiplicities \(m_e\).  Its support graph \(S=\supp(F)\) is the simple graph obtained by replacing each nonzero multiplicity by one edge.  Put \(w_e=2^{m_e}-1\) for every support edge \(e\).  If \(J\subseteq E(S)\), write \(w(J)=\prod_{e\in J}w_e\).  Then
\[
  \ec(F)=
  \sum_{\substack{J\subseteq E(S)\\ J\text{ covers }V(F)}}w(J).
\]
Indeed, whenever a support edge of multiplicity \(m\) is used, one may choose any nonempty subset of its \(m\) parallel copies.

\subsection{Entropy decomposition}
\label{sec:entropy-decomposition}

We now describe the general entropy mechanism used throughout the proof.  Fix a partition \(V(F)=\bigsqcup_{B\in\mathcal B}B\).
The parts of the partition are called regions.  A support edge with endpoints in two different regions is a boundary edge; all other support edges are internal to their unique region.
Let \(\partial\mathcal B\) be the set of all boundary support edges, and let \(\partial B\) be the set of boundary support edges incident to \(B\).

Let \(B\) be a region.  For \(U\subseteq B\), define the internal partition function
\[
  Z_B(U)=
  \sum_{\substack{J\subseteq E(S[B])\\ J\text{ covers }U}} w(J),
\]
where \(S[B]\) is the subgraph of \(S\) induced by \(B\).  Thus \(Z_B(U)\) is the total internal weight available when precisely the vertices in \(U\) still need to be covered inside the region.
Vertices outside \(U\) may also be incident to selected internal edges.

For a boundary state \(\eta\subseteq\partial B\), let \(V_B(\eta)\) be the set of vertices of \(B\) hit by \(\eta\).  Define the conditional partition function \(A_B(\eta)=w(\eta)^{1/2}Z_B\bigl(B\setminus V_B(\eta)\bigr)\).
The square root assigns one half of every selected boundary-edge weight to each of its two incident regions.

If \(\omega\subseteq\partial\mathcal B\) is a global boundary state, it records which boundary support edges are selected; write \(\omega_B=\omega\cap\partial B\) for the local boundary state seen by \(B\).  Once \(\omega\) is fixed, the remaining covering constraints are independent across regions.  Indeed, an edge cover \(J\) decomposes uniquely as \(\omega=J\cap\partial\mathcal B\) and \(J_B=J\cap E(S[B])\), where \(J_B\) covers \(B\setminus V_B(\omega_B)\), and conversely any such choices give an edge cover.  Moreover \(w(\omega)\prod_B w(J_B)=\prod_B\bigl(w(\omega_B)^{1/2}w(J_B)\bigr)\).
Hence
\begin{equation}
\label{eq:region-factorization}
  \ec(F)=\sum_\omega\prod_{B\in\mathcal B} A_B(\omega_B).
\end{equation}

We use the Gibbs variational formula: for every nonnegative function \(f\not\equiv0\) on a finite set,
\[
  \log\sum_x f(x)
  =\sup_\mu\left\{\E_\mu\log f(\mathbf X)+\Ent(\mathbf X)\right\},
\]
where the supremum is over all probability measures \(\mu\) on that finite set and \(\mathbf X\) has law \(\mu\).  Measures charging points with \(f=0\) contribute value \(-\infty\); equivalently, one may restrict to the support of \(f\).  See, for example, \cite{thomas2006elements,georgii2011gibbs}; for a recent exposition, see \cite{tao2023gibbs}.  We define the local free energy of a region by
\[
  \Phi_B
  =\sup_{\nu\in\mathcal P(2^{\partial B})}
  \left\{\E_\nu\log A_B(\bm\eta)+\frac12\Ent(\bm\eta)\right\},
\]
where \(\bm\eta\) denotes a random local boundary state with law \(\nu\).  Here and below, we use the same convention that measures charging states with \(A_B=0\) have value \(-\infty\).  The factor \(1/2\) reflects the fact that every boundary variable is shared by two regions.

\begin{lemma}
\label{lem:entropy-decomposition}
For every partition into regions, \(\log\ec(F)\leq\sum_{B\in\mathcal B}\Phi_B\).
\end{lemma}

\begin{proof}
Apply the Gibbs variational formula to the region factorization above, on the finite state space \(2^{\partial\mathcal B}\).  Thus \(\mu\) ranges over all probability measures on global boundary states, and \(\bm\omega\) denotes a random global boundary state with law \(\mu\).  Then
\[
  \log\ec(F)
  =\sup_\mu\left\{\Ent(\bm\omega)+\sum_B\E_\mu\log A_B(\bm\omega_B)\right\}.
\]
For every law \(\mu\), Shearer's inequality in fractional-cover form gives \(\Ent(\bm\omega)\leq\frac12\sum_B\Ent(\bm\omega_B)\) \cite{shearer1985problem},
because each coordinate of \(\bm\omega\), namely each boundary edge, appears in exactly two coordinate projections \(\bm\omega_B\).  Therefore the expression inside the supremum is at most
\[
  \sum_B\left(\E_\mu\log A_B(\bm\omega_B)+\frac12\Ent(\bm\omega_B)\right)
  \leq \sum_B\Phi_B.
\]
Taking the supremum over \(\mu\) proves the lemma.
\end{proof}

The variational definition has an equivalent partition-function form.

\begin{lemma}
\label{lem:local-free-energy}
For a region \(B\), set \(\FE_B=\sum_{\eta\subseteq\partial B}A_B(\eta)^2\).
Then \(\Phi_B=\frac12\log\FE_B\).
\end{lemma}

\begin{proof}
Apply the Gibbs variational formula to \(f(\eta)=A_B(\eta)^2\), and divide by two.
If \(\partial B=\varnothing\), this is the same statement with a single boundary state, giving \(\Phi_B=\log Z_B(B)\).
\end{proof}

\begin{remark}[Aggregation]
\label{rem:aggregation}
Suppose that boundary support edges of multiplicities \(m_1,\ldots,m_j\) are all incident to the same vertex of one fixed region.  Since
\[
  \FE_B=\sum_{\eta\subseteq\partial B}
  w(\eta)Z_B\bigl(B\setminus V_B(\eta)\bigr)^2,
\]
the total local weight of nonempty boundary states hitting that vertex is
\(\prod_{i=1}^j(1+(2^{m_i}-1))-1=2^{m_1+\cdots+m_j}-1\).
Thus, for the purpose of evaluating this single local free-energy partition function, these boundary edges may be aggregated into one effective boundary option whose multiplicity is \(m_1+\cdots+m_j\).  The global boundary variables in \eqref{eq:region-factorization} remain the original support edges.
\end{remark}

\subsection{Local estimates and odd uniformity}
\label{sec:local-regions}

We first analyze the basic local region consisting of two adjacent vertices.

\begin{lemma}
\label{lem:two-vertex-region}
Let \(uv\) be a support edge of multiplicity \(m\), where \(1\leq m\leq r-1\).  Suppose that the remaining incident multiplicities at each of \(u\) and \(v\) sum to \(r-m\).  For the two-vertex region \(B=\{u,v\}\), we have \(\Phi_B\leq\frac12\log T_r\).
Equality in this local estimate can occur only when \(m=1\) or \(m=r-1\).
\end{lemma}

\begin{proof}
Put \(p=2^m-1\) and \(h=2^{r-m}-1\).
After the local aggregation of \Cref{rem:aggregation}, each endpoint has a single effective boundary option of weight \(h\).  There are four effective hit-patterns.  If neither endpoint is hit from the boundary, the internal edge must be selected, giving conditional weight \(p\).  If exactly one endpoint is hit, the internal edge is still needed to cover the other endpoint, again giving weight \(p\).  If both endpoints are hit, the internal edge may be either absent or present, giving conditional weight \(1+p\).  Hence \(\FE_B=p^2+2hp^2+h^2(1+p)^2\).
Substituting \(p=2^m-1\) and \(h=2^{r-m}-1\), we get
\[
  T_r-\FE_B
  =2\bigl(2^{r-1}+2-2^m-2^{r-m}\bigr)
  =4(2^{m-1}-1)(2^{r-m-1}-1).
\]
This is nonnegative for \(1\leq m\leq r-1\), with equality only when \(m=1\) or \(m=r-1\).  By \Cref{lem:local-free-energy}, \(\Phi_B\leq \frac12\log T_r\).
\end{proof}

We shall use the following standard consequence of fractional matching theory \cite{schrijver2003combinatorial}.

\begin{lemma}
\label{lem:p2m}
Let \(S\) be a finite simple graph with a fractional perfect matching.  Then \(S\) contains a spanning subgraph whose components are single edges and odd cycles.
\end{lemma}

\begin{proof}
Consider the polytope
\(P=\{x\in\R_{\geq0}^{E(S)}:x(E_S(v))=1\text{ for every }v\in V(S)\}\),
where \(x(E_S(v))\) is the total weight of the edges incident to \(v\).
It is nonempty by assumption.  Moreover, \(P\) is the face of the
fractional matching polytope
\(\{x\in\R_{\geq0}^{E(S)}:x(E_S(v))\leq1\text{ for every }v\in V(S)\}\)
on which all vertex constraints are tight.  Hence an extreme point of
\(P\) is an extreme point of the fractional matching polytope.  Choose
such an extreme point.  By the half-integrality theorem for the
fractional matching polytope, its positive support is a disjoint union
of edges with value \(1\) and odd cycles with value \(1/2\).  Since
\(x(E_S(v))=1\) for every vertex, this support spans \(S\), and is
the required spanning subgraph.
\end{proof}

The next lemma handles the second type of region arising from such a spanning subgraph.

\begin{lemma}
\label{lem:open-cycle-region}
Let \(F\) be a loopless \(r\)-regular multigraph with maximum edge multiplicity at most \(r-1\).  Let \(C\) be an odd cycle in \(\supp(F)\), and use \(V(C)\) as a region, with all induced support edges internal.  Suppose that at least one support edge leaves \(V(C)\).  If each nonzero total outgoing multiplicity \(b_v=\sum_{e\ni v,\ e\not\subseteq V(C)}m_e\) is at most \(r-1\), then \(\Phi_C\leq|V(C)|\log T_r/4\).
\end{lemma}

\begin{proof}
By \Cref{lem:local-free-energy}, it is enough to bound \(\FE_C\).  We realize this local free-energy partition function as an ordinary edge-cover partition function on a doubled multigraph.

Take two copies of the induced region on \(V(C)\).  For each vertex \(v\in V(C)\), aggregate all outgoing boundary multiplicities at \(v\).  If \(b_v>0\), add an edge of multiplicity \(b_v\) joining the two copies of \(v\).  By \Cref{rem:aggregation}, summing over the common boundary state in \(\FE_C\) is exactly the same as deciding whether to use these added edges.  Therefore
\[
  \FE_C=
  \sum_{D\subseteq V(C)}
  \left(\prod_{v\in D}(2^{b_v}-1)\right)
  Z_C(V(C)\setminus D)^2,
\]
and this is precisely the edge-cover partition function of the doubled multigraph, written in the support-edge weighted form, so \(\FE_C=\ec(\widehat C)\), where \(\widehat C\) is the doubled multigraph.

The multigraph \(\widehat C\) is loopless and \(r\)-regular: at each copy of \(v\), the induced internal multiplicity is \(r-b_v\), and the edge joining the two copies contributes the remaining multiplicity \(b_v\).  Its maximum multiplicity is at most \(r-1\): this is true for the original support edges, and it is true for the added edges by hypothesis.  Moreover, the support of \(\widehat C\) has a perfect matching.  Choose a vertex \(v\) with \(b_v>0\), use the edge joining the two copies of \(v\), and match the two even paths obtained by deleting \(v\) from the two copies of the odd cycle, ignoring any chords.  Using the edges of this perfect matching as two-vertex regions, \Cref{lem:entropy-decomposition,lem:two-vertex-region} gives
\(\ec(\widehat C)\leq T_r^{|V(\widehat C)|/4}=T_r^{|V(C)|/2}\).
Thus \(\Phi_C=\frac12\log\FE_C\leq |V(C)|\log T_r/4\).
\end{proof}

\begin{proof}[Proof of \Cref{thm:odd-r}]
Let \(F\) be a loopless \(r\)-regular multigraph with maximum multiplicity at most \(r-1\), where \(r\) is odd.  By \Cref{lem:duality}, it suffices to prove \eqref{eq:main-edge-cover-bound}.

The support graph \(S\) has a fractional perfect matching, namely \(x_e=m_e/r\), since the incident multiplicities at each vertex sum to \(r\).  By \Cref{lem:p2m}, choose a spanning subgraph of \(S\) whose components are single edges and odd cycles.  Use the vertex sets of its components as the regions, and include all support edges induced inside each component as internal edges.

A single-edge component is handled by \Cref{lem:two-vertex-region}.  Now let \(C\) be an odd-cycle component.  If no support edge leaves \(V(C)\), then the induced multigraph \(F[V(C)]\), including any chords and all parallel copies, is an \(r\)-regular connected component on the odd number \(|V(C)|\) of vertices.  This is impossible by the handshaking identity \(2|E(F[V(C)])|=r|V(C)|\), since both \(r\) and \(|V(C)|\) are odd.  Thus \(C\) has nonempty boundary.  Also, at each vertex of \(C\), the two cycle support edges already have positive multiplicity, so \(b_v=r-\sum_{e\ni v,\ e\subseteq V(C)}m_e\leq r-2\).
Hence \Cref{lem:open-cycle-region} gives \(\Phi_C\leq|V(C)|\log T_r/4\).

Summing the regional bounds with \Cref{lem:entropy-decomposition}, we obtain \(\log\ec(F)\leq|V(F)|\log T_r/4\).
The theorem follows from \Cref{lem:duality}.
\end{proof}

\section{Computation for the case \texorpdfstring{$(r,d)=(4,2)$}{(r,d)=(4,2)}}\label{sec:r4}

In this section we prove \Cref{thm:r4}, keeping the entropy decomposition of \Cref{sec:entropy} throughout.  We used the parity of \(r\) in only one place, to rule out an odd-cycle region with empty boundary.  For \(r=4\) such a region can occur, so we will instead cut it into smaller regions.  Only one new type of region arises this way, a root vertex with a few matched pairs absorbed around it, and for these regions we will resort to a finite computation rather than an inequality.  Throughout this section \(T_4=211\).

\begin{lemma}
\label{lem:rooted-region-certificate}
Let \(1\leq\ell\leq4\), and let \(R\) be a region consisting of a root \(o\) together with \(\ell\) disjoint pairs \(\{a_i,b_i\}\).  Put \(U=\{a_1,b_1,\ldots,a_\ell,b_\ell\}\), and suppose that the support edges internal to \(R\) are the following.
\begin{enumerate}[(i)]
  \item The edges from \(o\) to \(U\).  Writing \(m_{ox}\) for the multiplicity of \(ox\), with the convention \(m_{ox}=0\) when \(ox\) is not an edge, these satisfy \(m_{ox}\leq3\), \(\sum_{x\in U}m_{ox}=4\), and \(m_{oa_i}+m_{ob_i}>0\) for every \(i\).
  \item The \(\ell\) pair edges \(a_ib_i\), of multiplicity at most \(3\).  We write \(m_{i(x)}\) for the multiplicity of the pair edge at \(x\in U\).
  \item At most one further edge \(xy\) between any two vertices of \(U\) lying in different pairs, of multiplicity \(m_{xy}\leq3\).
\end{enumerate}
Suppose also that \(o\) is incident to no boundary edge, that
\(m_{ox}+m_{i(x)}+\sum_y m_{xy}\leq4\) for every \(x\in U\), and that the total
boundary multiplicity at \(x\) is the slack
\(L_x=4-m_{ox}-m_{i(x)}-\sum_y m_{xy}\).
Then \(\Phi_R\leq(2\ell+1)\log T_4/4\).
\end{lemma}

\begin{proof}
By \Cref{lem:local-free-energy}, the assertion is equivalent to \(\FE_R^2\leq T_4^{2\ell+1}\).  We first put \(\FE_R\) in a form that a computer can evaluate exactly.  This part is bookkeeping, and the content of the lemma lies in the enumeration afterwards.

Let \(E_R(o)\) be the set of internal edges at the root, and let \(E_{\mathrm{nr}}\) be the set of remaining internal edges, that is, the pair edges together with the further edges in (iii).  After the aggregation of \Cref{rem:aggregation}, each \(x\in U\) has a single boundary option of weight \(2^{L_x}-1\), so a boundary state is a subset \(D\subseteq U\) of non-root vertices hit from outside, of weight \(\prod_{x\in D}(2^{L_x}-1)\).  Since the root has no boundary edge, it must be covered by a nonempty \(J_o\subseteq E_R(o)\).  Write \(V_U(J_o)\) for the set of vertices of \(U\) hit by \(J_o\).  The vertices that then remain to be covered inside \(R\) are those of \(U\setminus(D\cup V_U(J_o))\).  For \(U'\subseteq U\), set
\[
  Z_R(U')=
  \sum_{\substack{J\subseteq E_{\mathrm{nr}}\\ J\text{ covers }U'}} w(J).
\]
Note that the vertices outside \(U'\) may be covered by \(J\) as well.  Only the vertices of \(U'\) are required to be.
The internal partition function of \(R\) at boundary state \(D\) is then
\begin{equation}\label{eq:rooted-A}
  \widetilde Z_R(D)=
  \sum_{\varnothing\ne J_o\subseteq E_R(o)}
    w(J_o)\,Z_R\bigl(U\setminus(D\cup V_U(J_o))\bigr),
\end{equation}
so that \(A_R(D)=\bigl(\prod_{x\in D}(2^{L_x}-1)\bigr)^{1/2}\widetilde Z_R(D)\) and
\begin{equation}
\label{eq:rooted-FE}
  \FE_R=
  \sum_{D\subseteq U}
    \left(\prod_{x\in D}(2^{L_x}-1)\right) \widetilde Z_R(D)^2.
\end{equation}

Up to symmetries, the regions satisfying (i)--(iii) form a finite list.  Indeed \(\ell\leq4\) by hypothesis, and the degree caps at the vertices of \(U\) bound both the multiplicities and the number of edges allowed in (iii).  The list is far too long to inspect by hand, so we have verified the inequality by having a computer program run through it and evaluate \eqref{eq:rooted-A} and \eqref{eq:rooted-FE} in exact integer arithmetic.  We describe the enumeration in \Cref{app:rooted-certificate}, where we also record the resulting maxima of \(\FE_R\) for \(\ell=1,2,3,4\) and reproduce the code.  Each of these maxima has square strictly smaller than \(T_4^{2\ell+1}\), which is what we need.  The margin is not generous.  It is smallest at \(\ell=2\), where the ratio \(\FE_R^2/T_4^{5}\) is about \(0.9945\).
\end{proof}

\begin{proof}[Proof of \Cref{thm:r4}]
Let \(F\) be a loopless \(4\)-regular multigraph with maximum multiplicity at most \(3\).  As before, the support graph has a fractional perfect matching \(x_e=m_e/4\).  By \Cref{lem:p2m}, choose a spanning subgraph of the support whose components are single edges and odd cycles.  Single-edge components are handled by \Cref{lem:two-vertex-region}.  For an odd-cycle component with nonempty boundary, the two cycle edges at each vertex already have positive multiplicity, so the outgoing multiplicity is at most \(2\) and \Cref{lem:open-cycle-region} applies.

It remains to handle a closed odd-cycle component \(C\).  If \(C\) has no chord, then the equations \(m_{i-1}+m_i=4\) around the odd cycle force all cycle multiplicities to be \(2\).  Thus every cycle support edge has weight \(3\).  Recording whether the previous edge is selected gives the transfer matrix
\[
  K=\begin{pmatrix}0&3\\1&3\end{pmatrix}.
\]
Then its edge-cover partition function is \(\tr K^k\), where \(k=|V(C)|\) is the cycle length.  The eigenvalues are \(\kappa_\pm=(3\pm\sqrt{21})/2\), and \(\kappa_-<0\), so \(\tr K^k=\kappa_+^k+\kappa_-^k<\kappa_+^k\) for odd \(k\).  Note that \(\kappa_+^4=206.6\ldots<T_4\), again with little to spare.  Since this region has no boundary, \(\Phi_C\) is the logarithm of the partition function, so \(\Phi_C\leq|V(C)|\log T_4/4\).

Suppose now that \(C\) has a chord, meaning a support edge induced by \(V(C)\) but not belonging to the chosen cycle.  This is the case that parity ruled out for odd \(r\).  Choose a root vertex \(o\) incident to a chord.  Deleting \(o\) from the cycle leaves an even path, and we fix the perfect matching of this path into consecutive pairs.  We now refine the partition of \(V(C)\).  Let \(R\) consist of \(o\) together with precisely those matched pairs that contain at least one support-neighbour of \(o\), and keep every remaining matched pair as a separate two-vertex region.  All support edges crossing these refined parts are boundary edges for the final partition, even though they lie inside the original closed component.

The region \(R\) meets the hypotheses of \Cref{lem:rooted-region-certificate}, with the absorbed matched pairs playing the role of the pairs \(\{a_i,b_i\}\).  Indeed, \(o\) has support degree at most \(4\), since its multiplicities sum to \(4\), so at most four pairs are absorbed, and at least one is absorbed because \(o\) lies on a chord.  Every support edge at \(o\) has its other endpoint in an absorbed pair and is therefore internal to \(R\), so the internal multiplicities at \(o\) sum to \(4\).  Since the component is closed, \(o\) has no boundary edge in the refined partition.  Each absorbed pair contains a support-neighbour of \(o\) by construction, which gives (i), the bound \(m_{ox}\leq3\) being the global multiplicity bound.  The support graph is simple, so the matched edge is the only internal edge inside an absorbed pair, which gives (ii).  Every other internal edge of \(R\) therefore joins two different absorbed pairs, which gives (iii).  The edges of this last kind are the chords among absorbed non-root vertices together with the path edges not chosen as matched edges.  Finally, note that for a non-root vertex \(x\) of \(R\), the slack \(L_x\) is exactly the total multiplicity of the support edges from \(x\) to vertices outside \(R\).  This is where we use that \(F\) is \(4\)-regular.  All of these edges are boundary edges of the refined partition, and so \Cref{lem:rooted-region-certificate} gives \(\Phi_R\leq(2\ell+1)\log T_4/4=|R|\log T_4/4\).

Every unabsorbed matched pair \(B=\{u,v\}\) is used as a two-vertex region.  Its matched edge has some multiplicity \(1\leq m\leq3\), and is the only support edge internal to \(B\) by simplicity of the support graph, so every other support edge at \(u\) or \(v\) crosses the refined partition.  Thus the remaining incident multiplicity at each endpoint is \(4-m\), and \Cref{lem:two-vertex-region} gives \(\Phi_B\leq\frac12\log T_4=|B|\log T_4/4\).
Since \(R\) and the unabsorbed pairs partition \(V(C)\), these two estimates sum to
\(\sum_P\Phi_P\leq|V(C)|\log T_4/4\), where \(P\) runs over the refined regions of the component.

Now perform this refinement on every closed odd-cycle component with a chord, leaving the remaining components as they are, and apply \Cref{lem:entropy-decomposition} to the resulting partition.  Combining the regional estimates gives \(\log\ec(F)\leq|V(F)|\log T_4/4\), so \Cref{conj:bbn} holds for \((r,d)=(4,2)\) by \Cref{lem:duality}.
\end{proof}

\section{Concluding remarks}

It remains open whether \(H_{r,d}\) is exactly optimal among cross-edge-free hypergraphs, even for \(r=3\).  It would also be interesting to replace the finite certificate of \Cref{lem:rooted-region-certificate} by a conceptual inequality.  More broadly, the local-free-energy perspective seems worth pursuing beyond the degree-two setting treated here.

\section*{Acknowledgements}
Part of this work began in Spring 2023 during Math 496, Research in Mathematics, at Rutgers University.  Zeyu Zheng thanks Bhargav Narayanan for supervising the initial project and for helpful early discussions. Independently, Sarantis and Tetali had collaborated with Will Perkins on the Balogh-Bollob\'as-Narayanan conjecture, and thank Perkins for his insight on using the occupancy method to tackle the conjecture.

\section*{Disclosure of AI Use}
The first version of this paper, which contained a proof of a weaker form of \Cref{thm:cross-edge-free}, was written without the use of AI tools.  In the course of expanding the paper, the authors used AI assistants (several versions of Claude, Cursor, and GPT) to explore and test candidate proof strategies and to write and check numerical computations, including the verification program of \Cref{app:rooted-certificate}.

\printbibliography

\newpage
\appendix

\newpage
\section{Proof of \Cref{prop:counterex}}
\label{sec:appendix-counterexample}

\begin{proof}[Proof of \Cref{prop:counterex}]
Fix \(r\geq3\), and put \(s=r-1\).  Label the marked vertices of
\(H_{r,d}\) as \(u_1,\ldots,u_d\) and its unmarked blocks as
\(P_1,\ldots,P_d\), each of size \(s\).  After relabelling, the deleted
matching consists of the edges \(\{u_i\}\cup P_i\).  Thus the edges of
\(\widetilde H_{r,d}\) are precisely \(\{u_i\}\cup P_j\) with \(i\neq j\),
and \(\widetilde H_{r,d}\) is \(r\)-uniform and \((d-1)\)-regular on
\(rd\) vertices.  Any three pairwise-intersecting edges share either
their marked vertex or their unmarked block.  In either case, all
three pairwise intersections equal the common intersection, so
\(\widetilde H_{r,d}\) has no cross-edges.

Put
\[
    b(\lam)=(1+\lam)^s-\lam^s,
\]
the partition function of one unmarked block when full occupation is
forbidden.  We compute the partition function of \(\widetilde H_{r,d}\)
by conditioning on the occupied marked vertices.  If none is occupied,
all unmarked vertices are unrestricted.  If exactly \(u_i\) is
occupied, the block \(P_i\) is unrestricted and every other block must
be nonfull.  If at least two marked vertices are occupied, every block
must be nonfull.  Consequently,
\begin{align}
    Z_{\widetilde H_{r,d}}(\lam)
    &=(1+\lam)^{sd}
      +d\lam(1+\lam)^s b(\lam)^{d-1}\notag\\
    &\quad+\big((1+\lam)^d-1-d\lam\big)b(\lam)^d\notag\\
    &=\big((1+\lam)b(\lam)\big)^d
      +(1+\lam)^{sd}-b(\lam)^d\notag\\
    &\quad+d\lam^r b(\lam)^{d-1}.
    \label{eq:counterexample-partition}
\end{align}
For comparison, for every \(k\geq1\),
\[
    Z_{H_{r,k}}(\lam)
    =(1+\lam)^{sk}
      +\big((1+\lam)^k-1\big)b(\lam)^k.
\]

To compare occupancies at \(\lam=1\), write
\[
    q=2^s-1,\qquad
    \rho=\frac{q+1}{2q},\qquad
    \beta=\frac12-\frac{s}{2rq},\qquad
    \gamma=\frac{q-s}{2rq}.
\]
Since \(s\geq2\), we have \(q>s\), \(\gamma>0\), and
\(1/2<\rho<1\).  Using \(b(1)=q\) and
\(b'(1)=s(q-1)/2\), differentiation of the partition functions above
gives
\begin{align}
    \alpha_{H_{r,k}}(1)
    &=\beta+
    \frac{-\gamma\rho^k+2^{-k}/(2r)}
         {1+\rho^k-2^{-k}},
    \label{eq:counterexample-benchmark-occupancy}\\[1ex]
    \alpha_{\widetilde H_{r,d}}(1)
    &=\beta+
    \frac{-\gamma\rho^d+
      \dfrac{2^{-d}}{2r}\left(1+\dfrac{s+2-d}{q}+\dfrac{s}{q^2}\right)}
         {1+\rho^d+\left(\dfrac{d}{q}-1\right)2^{-d}}.
    \label{eq:counterexample-occupancy}
\end{align}
Here the normalizing vertex counts are \(rk\) and \(rd\),
respectively.

For fixed \(r\), both denominators tend to \(1\) as \(d\to\infty\)
with \(k=d-1\).  Moreover, \(d2^{-d}=o(\rho^{d-1})\), since
\(\rho>1/2\).  Subtracting
\eqref{eq:counterexample-benchmark-occupancy} from
\eqref{eq:counterexample-occupancy} therefore gives
\[
    \lim_{d\to\infty}
    \frac{\alpha_{\widetilde H_{r,d}}(1)-\alpha_{H_{r,d-1}}(1)}
         {\rho^{d-1}}
    =\gamma(1-\rho)
    =\frac{(q-s)(q-1)}{4rq^2}>0.
\]
It follows that \(\alpha_{\widetilde H_{r,d}}(1)>
\alpha_{H_{r,d-1}}(1)\) for all sufficiently large \(d\).
Since these two hypergraphs are both \(r\)-uniform and
\((d-1)\)-regular, this disproves \Cref{conj:bbn_general} (i)
for every \(r\geq3\).

For a concrete instance, taking \(r=3\) and \(d=10\) in the exact
formulas above gives
\[
    \alpha_{\widetilde H_{3,10}}(1)-\alpha_{H_{3,9}}(1)
    =\frac{15440593977}{636263820007681}>0.
\]
\end{proof}

\newpage
\section{The finite certificate}
\label{app:rooted-certificate}

\RestyleAlgo{ruled}

\subsection{Algorithm}
\label{app:rooted-certificate:algo}

This appendix records the finite verification used for the rooted-region
certificate.  Put \(U=\{a_1,b_1,\ldots,a_\ell,b_\ell\}\).  A boundary state is a
subset \(D\subseteq U\) of non-root vertices hit from outside.  The value of
\(\FE_R\) is invariant under relabelling of the matched pairs and of the two
endpoints inside each pair.  For \(\ell\leq3\), the program enumerates all labelled
root patterns and uses the pair relabelling symmetry to put
\(m_1,\ldots,m_\ell\) in nondecreasing order.  For \(\ell=4\), the total root
multiplicity is \(4\) and every matched pair must be hit by the root, so each
pair receives root multiplicity exactly \(1\).  After first permuting the pairs
so that \(m_1\leq\cdots\leq m_4\) and then swapping endpoints within each pair,
every such root pattern is represented by
\((1,0,1,0,1,0,1,0)\).  The verification is as follows.

\begin{algorithm}[h]
\caption{Verification of \Cref{lem:rooted-region-certificate}.}
\KwIn{A number \(\ell\in\{1,2,3,4\}\) of matched pairs.}
\ForEach{root pattern \((m_{ox})_{x\in U}\) with total root multiplicity \(4\), every pair hit, and \(m_{ox}\in\{0,1,2,3\}\)}{
  \ForEach{nondecreasing pattern \(1\leq m_1\leq\cdots\leq m_\ell\leq3\)}{
    \If{\(m_{ox}+m_{i(x)}>4\) for some \(x\in U\)}{discard this pattern\;}
    Enumerate all additional non-root edges between different matched pairs, with multiplicities allowed by the remaining degree caps\;
    \ForEach{resulting legal rooted region \(R\)}{
      Compute \(L_x\), all \(Z_R(U')\), all \(\widetilde Z_R(D)\), and \(\FE_R\) from \eqref{eq:rooted-FE}\;
      Record the maximum and check \((\max\FE_R)^2<T_4^{2\ell+1}\)\;
    }
  }
}
\end{algorithm}

In the step enumerating additional non-root edges, the program considers every
unordered pair of endpoints lying in different matched pairs as a possible
support edge, assigns it multiplicity \(0,1,2,3\), and keeps precisely the
choices allowed by the remaining degree caps.  Thus every rooted region
satisfying the hypotheses is represented, up to the relabellings above.  For
each represented region, the program computes the weights of all subsets of
non-root internal edges by their covered vertex set and applies a subset zeta
transform to obtain all \(Z_R(U')\).  It then enumerates all nonempty subsets of
root support edges to compute \(\widetilde Z_R(D)\), and finally evaluates
\(\FE_R\) from \eqref{eq:rooted-FE} using exact integer arithmetic.

For \(\ell=4\), the program fixes the representative
\((1,0,1,0,1,0,1,0)\).  The resulting certificate gives the maxima for
\(\ell=1,2,3\) and, for \(\ell=4\), the maxima for each nondecreasing
matched-edge pattern:
\[
\begin{array}{c|c@{\qquad}c|c@{\qquad}c|c}
(m_1,m_2,m_3,m_4)&\max\FE_R&
(m_1,m_2,m_3,m_4)&\max\FE_R&
(m_1,m_2,m_3,m_4)&\max\FE_R\\ \hline
1111&26730624288&1112&26473597530&1113&27070292187\\
1122&26219052000&1123&26810009295&1133&27413447163\\
1222&25966963530&1223&26552239155&1233&27149875035\\
1333&27760096317&2222&25717308192&2223&26296957287\\
2233&26888847507&2333&27493201797&3333&28110245535
\end{array}
\]
Taking the maximum over the printed \(\ell=4\) lines gives the following summary:
\[
\begin{array}{c|c}
\ell&\max_R\FE_R\\ \hline
1&3007\\
2&644925\\
3&132799423\\
4&28110245535
\end{array}
\]
Each displayed maximum has square strictly smaller than \(T_4^{2\ell+1}\), which
is the required certificate inequality.

\subsection{C++ implementation}
\label{app:rooted-certificate:cpp}

The complete \texttt{C++17} source is reproduced below.  The type
\texttt{\_\_int128\_t} is used for exact integer arithmetic in the certificate.

\begin{lstlisting}[language=C++]
#include <algorithm>
#include <array>
#include <iostream>
#include <string>
#include <utility>
#include <vector>

using i128 = __int128_t;

namespace {

constexpr int kDegree = 4;
constexpr int kMaxMult = 3;
constexpr int kMaxPairs = 4;
constexpr int kTarget = 211;
constexpr int kMaxVertices = 2 * kMaxPairs;

constexpr std::array<int, 5> kWeight = {0, 1, 3, 7, 15};

using MatchArray = std::array<int, kMaxPairs>;
using VertexArray = std::array<int, kMaxVertices>;

i128 Pow(i128 base, int exp) {
  i128 ans = 1;
  while (exp-- > 0) ans *= base;
  return ans;
}

std::string Decimal(i128 x) {
  if (x == 0) return "0";
  bool neg = x < 0;
  if (neg) x = -x;
  std::string out;
  while (x > 0) {
    out.push_back(static_cast<char>('0' + x % 10));
    x /= 10;
  }
  if (neg) out.push_back('-');
  std::reverse(out.begin(), out.end());
  return out;
}

struct Edge {
  int u = 0;
  int v = 0;
  int mult = 0;
};

// Implements the manuscript formulas
//   Xi_R = sum_D prod_{x in D}(2^{L_x}-1) Ztilde_R(D)^2,
//   Ztilde_R(D) = sum_{nonempty J_o} w(J_o) Z_R(U \ (D union V_U(J_o))).
i128 LocalFreeEnergyPartition(int pairs, const VertexArray& root,
                              const MatchArray& matched,
                              const std::vector<Edge>& extra_edges) {
  const int n = 2 * pairs;
  const int states = 1 << n;
  const int all = states - 1;

  std::vector<i128> z_by_need(states, 0);
  z_by_need[0] = 1;

  auto add_internal_edge = [&](int cover_mask, int weight) {
    std::vector<i128> old = z_by_need;
    for (int mask = 0; mask < states; ++mask) {
      if (old[mask] != 0) z_by_need[mask | cover_mask] += old[mask] * weight;
    }
  };

  VertexArray exposed{};
  exposed.fill(kDegree);
  for (int i = 0; i < pairs; ++i) {
    int u = 2 * i;
    int v = 2 * i + 1;
    int mult = matched[i];
    exposed[u] -= mult;
    exposed[v] -= mult;
    add_internal_edge((1 << u) | (1 << v), kWeight[mult]);
  }
  for (int v = 0; v < n; ++v) exposed[v] -= root[v];
  for (const Edge& e : extra_edges) {
    exposed[e.u] -= e.mult;
    exposed[e.v] -= e.mult;
    add_internal_edge((1 << e.u) | (1 << e.v), kWeight[e.mult]);
  }

  // Subset zeta transform: afterwards z_by_need[U'] = Z_R(U'), the total
  // weight of the internal edge sets covering U'.
  for (int bit = 0; bit < n; ++bit) {
    for (int mask = 0; mask < states; ++mask) {
      if ((mask & (1 << bit)) == 0) z_by_need[mask] += z_by_need[mask | (1 << bit)];
    }
  }

  std::vector<std::pair<int, int>> root_choices;  // (covered mask, weight)
  root_choices.push_back({0, 1});
  for (int v = 0; v < n; ++v) {
    if (root[v] == 0) continue;
    int weight = kWeight[root[v]];
    int current_size = static_cast<int>(root_choices.size());
    for (int i = 0; i < current_size; ++i) {
      root_choices.push_back({root_choices[i].first | (1 << v),
                              root_choices[i].second * weight});
    }
  }

  i128 total = 0;
  for (int boundary = 0; boundary < states; ++boundary) {
    i128 boundary_weight = 1;
    for (int v = 0; v < n; ++v) {
      if ((boundary >> v) & 1) boundary_weight *= kWeight[exposed[v]];
    }
    if (boundary_weight == 0) continue;

    i128 conditional = 0;
    for (int i = 1; i < static_cast<int>(root_choices.size()); ++i) {
      auto [covered_by_root, root_weight] = root_choices[i];
      int need = all & ~(boundary | covered_by_root);
      conditional += static_cast<i128>(root_weight) * z_by_need[need];
    }
    total += boundary_weight * conditional * conditional;
  }
  return total;
}

i128 MaxOverExtraEdges(int pairs, const VertexArray& root,
                       const MatchArray& matched) {
  const int n = 2 * pairs;
  std::vector<Edge> extra_edges;

  VertexArray degree_used{};
  for (int i = 0; i < pairs; ++i) {
    degree_used[2 * i] += matched[i];
    degree_used[2 * i + 1] += matched[i];
  }
  for (int v = 0; v < n; ++v) degree_used[v] += root[v];

  std::vector<std::pair<int, int>> edge_slots;
  for (int u = 0; u < n; ++u) {
    for (int v = u + 1; v < n; ++v) {
      if (u / 2 != v / 2) edge_slots.push_back({u, v});
    }
  }

  i128 maximum = 0;
  auto search = [&](auto&& self, int slot) -> void {
    if (slot == static_cast<int>(edge_slots.size())) {
      maximum = std::max(maximum,
                         LocalFreeEnergyPartition(pairs, root, matched, extra_edges));
      return;
    }

    auto [u, v] = edge_slots[slot];
    int max_mult = std::min({kMaxMult, kDegree - degree_used[u],
                             kDegree - degree_used[v]});
    for (int mult = 0; mult <= max_mult; ++mult) {
      if (mult > 0) {
        degree_used[u] += mult;
        degree_used[v] += mult;
        extra_edges.push_back({u, v, mult});
      }
      self(self, slot + 1);
      if (mult > 0) {
        extra_edges.pop_back();
        degree_used[u] -= mult;
        degree_used[v] -= mult;
      }
    }
  };
  search(search, 0);
  return maximum;
}

template <class Visit>
void ForMatchedPatterns(int pairs, Visit visit) {
  MatchArray matched{};

  auto search = [&](auto&& self, int pair, int least_mult) -> void {
    if (pair == pairs) {
      visit(matched);
      return;
    }
    for (int mult = least_mult; mult <= kMaxMult; ++mult) {
      matched[pair] = mult;
      self(self, pair + 1, mult);
    }
  };
  search(search, 0, 1);
}

bool CheckSmallRootedRegions() {
  std::cout << "pairs max_Xi\n";
  for (int pairs = 1; pairs <= 3; ++pairs) {
    VertexArray root{};
    i128 maximum = 0;

    auto enumerate_roots = [&](auto&& self, int vertex, int remaining) -> void {
      if (vertex == 2 * pairs) {
        if (remaining != 0) return;
        for (int i = 0; i < pairs; ++i) {
          if (root[2 * i] + root[2 * i + 1] == 0) return;
        }
        ForMatchedPatterns(pairs, [&](const MatchArray& matched) {
          bool legal = true;
          for (int v = 0; v < 2 * pairs; ++v) {
            legal &= root[v] + matched[v / 2] <= kDegree;
          }
          if (legal) {
            maximum = std::max(maximum, MaxOverExtraEdges(pairs, root, matched));
          }
        });
        return;
      }
      for (int mult = 0; mult <= std::min(kMaxMult, remaining); ++mult) {
        root[vertex] = mult;
        self(self, vertex + 1, remaining - mult);
      }
      root[vertex] = 0;
    };
    enumerate_roots(enumerate_roots, 0, kDegree);

    std::cout << pairs << " " << Decimal(maximum) << "\n";
    if (maximum * maximum >= Pow(kTarget, 2 * pairs + 1)) {
      std::cerr << "free-energy certificate failed for pairs=" << pairs << "\n";
      return false;
    }
  }
  return true;
}

bool CheckFourPairRootedRegions() {
  VertexArray root = {1, 0, 1, 0, 1, 0, 1, 0};

  std::cout << "(m1,m2,m3,m4) max_Xi\n";
  bool ok = true;
  ForMatchedPatterns(4, [&](const MatchArray& m) {
    i128 maximum = MaxOverExtraEdges(4, root, m);
    std::cout << m[0] << m[1] << m[2] << m[3] << " "
              << Decimal(maximum) << "\n";
    if (maximum * maximum >= Pow(kTarget, 9)) {
      std::cerr << "free-energy certificate failed for matched pattern "
                << m[0] << m[1] << m[2] << m[3] << "\n";
      ok = false;
    }
  });
  return ok;
}

}  // namespace

int main() {
  if (!CheckSmallRootedRegions()) return 1;
  if (!CheckFourPairRootedRegions()) return 1;
  return 0;
}
\end{lstlisting}

\end{document}